\documentclass [leqno]{amsart}
\usepackage {graphicx}
\usepackage{enumerate}
\usepackage{wrapfig}
\usepackage{amsmath,amssymb, amsthm}
\usepackage{xcolor}

 \usepackage[usenames,dvipsnames]{pstricks}

\numberwithin{equation}{section}

\newtheorem{theorem}{Theorem}[section]

\newtheorem{lemma}[theorem]{Lemma}
\newtheorem{corollary}[theorem]{Corollary}

\newtheorem*{theorem*}{Theorem}
\newtheorem{example}[theorem]{Example}

\newtheorem{remark}[theorem]{Remark}

\def\N{\mathbb{N}}

\def\P{\mathbf{P}}

\def\E{\mathbf{E}}

\def\H{\mathbf{H}}
\def\K{\mathbf{K}}
\def\HN{{\mathbf{H}}^{(N)}}

\def\vg{\boldsymbol{\vec{\gamma}}}

\begin{document}

\title{A limit theorem for selectors}
%: \\ Zermelo and Sperner statistics}
 \author[F. Durango, J.\,L. Fern\'{a}ndez, P. Fern\'{a}ndez and  M.\,J. Gonz\'{a}lez]{Francisco Durango, Jos\'{e} L. Fern\'{a}ndez, Pablo Fern\'{a}ndez and  Mar\'{\i}a J. Gonz\'{a}lez}

 \date{\today}


 \maketitle

 \begin{abstract}
 {Any (measurable) function $K$ from $\mathbb{R}^n$ to $\mathbb{R}$ defines an operator $\mathbf{K}$ acting on random variables $X$ by $\mathbf{K}(X)=K(X_1, \ldots, X_n)$, where the~$X_j$ are independent copies of~$X$. The main result of this paper concerns selectors $H$, continuous functions defined in~$\mathbb{R}^n$ and such that $H(x_1, x_2, \ldots, x_n) \in \{x_1,x_2, \ldots, x_n\}$. For each such selector~$H$ (except for projections onto a single coordinate) there is a unique point $\omega_H$ in the interval $(0,1)$  so that for any random variable $X$ the iterates $\mathbf{H}^{(N)}$ acting on $X$ converge in distribution as $N \to \infty$ to the $\omega_H$-quantile of $X$. %For the so called Zermelo statistics $Z$ (and also for order statistics), there  is a unique $\omega_Z$ so that $\mathbf{Z}^{(N)}(X)$ converges in distribution to the $\omega_{Z}$-quantile of~$X$.
 }
 \end{abstract}

\section{Introduction}

Any (Borel measurable) function $K$ from $\mathbb{R}^n$ to $\mathbb{R}$ defines an operator $\mathbf{K}$ acting on random variables $X$ by $\mathbf{K}(X)=K(X_1, \ldots, X_n)$, where the $X_j$ are independent copies of $X$. The~\mbox{$N$-th} iterate of the operator $\K$ is denoted by $\K^{(N)}$.
%$See Section \ref{section:limit theorem} for the precise definition of this iteration.
We investigate in this paper the convergence (in distribution) of the iterates $\K^{(N)}$ of a special kind of functions $K$ which we call selectors.

A \textbf{selector} is a \textit{continuous} function $H\colon \mathbb{R}^n\to\mathbb{R}$ which satisfies the \textit{selecting} property:
$$
 H(x_1, x_2, \ldots, x_n) \in \{x_1,x_2, \ldots, x_n\}\,, \quad \mbox{for any $(x_1, x_2, \ldots, x_n)\in \mathbb{R}^n$}\,.
$$
(Occasionally, in what follows, we shall consider also measurable --but not necessarily conti\-nuous-- functions $H$ satisfying the selecting property.) Maximum, minimum, and, in general, order statistics are conspicuous examples of selectors.

We will prove that selectors admit very concrete expressions combining max and min operators, and Sperner families: in fact, selectors are \textit{Sperner statistics}; see the definition in Section \ref{section:sperner}, and then Theorem~\ref{theor:continuous conservative equal selectors equal sperner}.

\smallskip

Theorem \ref{main th Sperner} claims that to each selector $H$ (except for projections onto a single coordinate) we may ascribe a unique point $\omega_H$ in the interval $[0,1]$  so that for \textit{any} random variable~$X$ the iterates~$\mathbf{H}^{(N)}$ acting on~$X$ converge in distribution to the quantile of $X$ corresponding to the point $\omega_H$. This $\omega_H$ is the unique fixed point of a certain polynomial~$h$ canonically associated to $H$.

%For the so called Zermelo statistics $Z$ and for the standard order statistics, there  is a \textit{unique} point $\omega_Z\in \mathcal{F}_Z$ so that $\mathbf{Z}^{(N)}(X)$ converges in distribution to the $\omega_{Z}$-quantile of $X$.

\smallskip
This limit result parallels the Weak Law of Large Numbers,  which  corresponds to the function $K(x_1, \ldots, x_n)=\frac{1}{n}\sum_{i=1}^n x_i$, and the Central Limit Theorem, which corresponds  to  the function $K(x_1, \ldots, x_n)=\frac{1}{\sqrt{n}}\sum_{i=1}^n x_i$ (for typified random variable $X$). See Section \ref{section:LGN}.

Theorem \ref{main th Sperner} is an outgrow of a convergence result for  the Zermelo value of binary games when the outcomes of the game are randomized and the length of the game tends to infinity. We discuss this illustration in  Section \ref{randomizingzermelo}, as a motivating starting point.

\medskip
The paper is organized as follows. We introduce
in Section \ref{section:conservative} the basic notions that will be used in the paper: conservative and Sperner statistics, and the associated modules. Section~\ref{section:selectors} shows the equivalence between selectors and  Sperner statistics. We introduce in Section \ref{sec:sperner polynomials} the so called Sperner polynomials, that will be essential
in the analysis (see Section~\ref{sec:fixed points of selectors}) of the fixed points of modules of selectors.
The convergence result for the iteration of selectors is proved in Section~\ref{section:limit theorem}.
Finally, Section \ref{section:LGN} discusses some analogies with laws of large numbers.

\subsection{Notation and preliminaries}\label{notation}

%For the sake of brevity, the set of positive integers $\{1, 2, \ldots, n\}$ will be denoted by $\mathbb{N}_n$.

\subsubsection*{Quantiles.} We denote the distribution function of a random variable $X$ by $F_X$. %The distribution function of a standard normal variable is denoted by $\Phi$.
We define the \textit{quantile function} $Q_X$ of $X$ as follows. $Q_X$ is defined in $[0,1]$. Write, for $\eta \in (0,1)$,
$$\begin{aligned}
a_X(\eta)&=\sup{\{t \in \mathbb{R}: F_X(t) <\eta\}}=\inf{\{t \in \mathbb{R}: F_X(t) \ge\eta\}}\, ,\\
z_X(\eta)&=\inf{\{t \in \mathbb{R}: F_X(t) >\eta\}}=\sup{\{t \in \mathbb{R}: F_X(t) \le\eta\}}\, .
\end{aligned}$$
Observe that $a_X(\eta) \le z_X(\eta)$ and that $\eta \le F_X\big(a_X(\eta)\big)\le F_X\big(z_X(\eta)\big)$. %\ojo{CREO QUE EL PRIMERO ES UN $=$ EN LUGAR DE UN $\le$.}
Notice also that
$a_X(\eta) < z_X(\eta)$ means that $\P\big(X\le a_X(\eta)\big)=\eta$, and that $\P\big(a_X(\eta) < X < z_X(\eta)\big)=0$.

\smallskip

For $\eta \in (0,1)$, we define the $\eta$-quantile $Q_X(\eta)$ of $X$
as the \textit{random variable} which takes the value $a_X(\eta)$ with probability $\eta$, and the value $z_X(\eta)$ with probability $1-\eta$.
Observe that,  if $a_X(\eta) < z_X(\eta)$, then  $Q_X$ takes two values, and that $Q_X$ is a constant if $a_X(\eta) = z_X(\eta)$.

We set also
$$
Q_X(0)=\inf{\{t\in \mathbb{R}: F_X(t)>0\}}\, ,\quad\text{and}\quad
Q_X(1)=\sup{\{t\in \mathbb{R}: F_X(t)<1\}}
$$
(the essential infimum of $X$,  and the essential supremum of $X$, respectively). Notice that $\P(Q_X(0)\le X \le Q_X(1))=1$.

\begin{itemize}
\item For a continuous random variable with continuous and \textit{strictly increasing} distribution function, $Q_X(\eta)$ is a constant for each $\eta \in (0,1)$, $Q_X(0)=-\infty$, $Q_X(1)=+\infty$, and $Q_X$ (restricted to $(0,1)$) is the inverse for $F_X$.
\item For a finite random variable $X$, the quantile $Q_X(\eta)$ is a constant, unless $\eta$ is one of the values attained by $F_X$. For instance, if $X$ takes just two values $a <b$ with respective probabilities $p \in(0,1)$ and $1-p$, then
    $$
    Q_X(\eta)\stackrel{\rm d}{=}\begin{cases}
    a, & \  \text{if} \ 0\le \eta <p\, , \\
    X, &  \  \text{if} \ \eta =p\, ,\\
        b, & \  \text{if} \ p< \eta \le 1\, . \\
    \end{cases}
    $$
 \end{itemize}

See \cite{EH} for a detailed description of quantiles, where $Q_X(\eta)$ is defined always as $a_X(\eta)$.

%The \textit{pseudo-inverse} of $F_X$ is denoted by $G_X$ and is defined in $[0,1]$ by
%$$
%u \in (0,1] \mapsto G_X(u)=\inf\{t \in \mathbb{R}; F_X(t)\ge u\}\, ,
%$$
%while $G_X(0)=\sup\{t \in \mathbb{R}; F_X(t)=0\}$. We use the convention that $\inf\emptyset=+\infty$. Observe that $G_X$ takes values in $[-\infty, +\infty]$.
%See \cite{EH}.

%\
%
%For each $u \in [0,1]$, the value $G_X(u)$ if the \textit{$u$-quantile of $X$}; for a variable $X$ with strictly increasing distribution function $\P(X \le G_x(u))=u$ and $\P(X \ge G_x(u))=1-u$.

\subsubsection*{Bernstein polynomials.} For each integer $n \ge 1$, the \textit{Bernstein polynomials}, given by
$$
B^{(n)}_j(t)=\binom{n}{j} \,t^j (1-t)^{n-j},\quad\text{for $j=0, 1, \ldots, n$,}
$$
form a basis of the space of polynomials of degree at most $n$. Recall that
\begin{align}
\label{Bernstein positive}\text{(positivity)}\quad&B^{(n)}_j(t)>0 \quad\text{for $t\in (0,1)$;}
\\
\label{Bernstein partition}\text{(partition of unity)}\quad&\sum_{j=0}^n B^{(n)}_j(t)=1.
\\
\label{equation:derivativeBernstein}\text{(derivatives)}\quad &
\frac{d}{dx}B^{(n)}_j(x)=n \big[B^{(n-1)}_{j-1}(x)-B^{(n-1)}_{j}(x)\big]\quad\text{for $0 \le j \le n$.}
\end{align}
(We are using here the convention that $B^{(n-1)}_{n}\equiv 0$ and that $B^{(n-1)}_{-1}\equiv 0$.)

\subsubsection*{Subsets and the Boolean cube.} Let $\mathbb{B}^n=\{0,1\}^n$ be the Boolean cube. We shall use the standard identification between $\mathbb{B}^n$ and $\mathcal{P}(n)$ (the subsets of $\{1,\dots,n\}$).

For each $p\in(0,1)$, the Bernoulli measure $\mu_p$ in $\mathbb{B}^n$ is given by $\mu_p(\{(x_1,\dots, x_n)\})=p^{\#\{x_j=1\}}\, (1-p)^{\#\{x_j=0\}}$ for any $(x_1,\dots, x_n)\in \mathbb{B}^n$; so the coordinates are independent Bernoulli random variables with success probability $p$.
%such that $\mu_p(x_j=1)=p$).
%\smallskip
% $\mu_p$ (the product probability in $\mathbb{B}^n$ such that $\mu_p(x_j=1)=p$).

A family $\mathcal{D}$ of subsets of $\{1,\dots,n\}$ is a \textit{downset} if the following property holds: if $A\in \mathcal{D}$ and $B\subseteq A$, then $B\in\mathcal{D}$. A collection $\mathcal{U}$ is an \textit{upset} if $A\in \mathcal{U}$ and $A\subseteq B$ implies $B\in\mathcal{U}$.

A \textit{Sperner family} in $\{1, 2, \ldots n\}$ is a collection $\{A_1, \ldots, A_k\}$ of nonempty subsets of  $\{1, 2, \ldots n\}$ such that no $A_i$ of the family is contained in any other  $A_j$ of the family. For instance, the family of all subsets of size $r$ ($1 \le r \le n$) is a Sperner family.

A family of nonempty \textit{pairwise} disjoint subsets of $\{1, 2, \ldots n\}$ will be called a \textit{disjoint family}; such a disjoint family is obviously a Sperner family.

%\smallskip
%For a family $\mathcal{A}\subset \mathcal{P}(n)$, we define its \textit{total size} as
%$$
%\|\mathcal{A}\|=\sum_{A\in\mathcal{A}} |A|.
%$$
%Letting
%$$
%b_k(\mathcal{A})=\#\{A\in\mathcal{A}: |A|=k\},
%$$
%we can write
%$$
%|\mathcal{A}|=\sum_{k=0}^n b_k(\mathcal{A})\quad\text{and}\quad \|\mathcal{A}\|=\sum_{k=0}^n k\, b_k(\mathcal{A}).
%$$

%We shall denote by $\{0,1\}^n$ the collection of the $2^n$ vectors of $\mathbb{R}^n$ all whose coordinates are $0$ or $1$. The functions defined in %$\{0,1\}^n$ with values in $\{0,1\}$ are called Boolean functions (of degree $n$).

%We denote by $\mathcal{Q}_n$  the collection of polynomials $g$ of degree at most $n$ which map the interval $[0,1]$ onto itself and satisfy $g(0)=0$ and %$g(1)=1$. The polynomials in $\mathcal{Q}_n$ are precisely those polynomials
%$$
%g(t)=\sum_{j=0}^n \gamma_j B^{(n)}_j(t)\,,
%$$
%with $\gamma_0=0$ (so that $g(0)=0$), $\gamma_n=1$ (so that $g(1)=1$) and  $0 \le \gamma_j \le 1$, for $0 \le j \le n$. It is obvious that if  last condition %implies  %that $g([0,1]) \subset [0,1]$. But the reciprocal is not direct, see \cite{QRR}.

\section{Randomizing Zermelo's (value of game) algorithm}\label{randomizingzermelo}

This research originated with the analysis of a randomized version of Zermelo's algorithm, which we describe now.

Two players $\alpha$ and $\beta$ alternately add symbols $L$ or $R$  to form a string; symbols are always added to the right of the existing string. The string is empty at the outset; the game ends when the length of the string is $2N$, for some predetermined integer $N \ge 1$. The collection of strings of length $2N$ with the symbols $L$ and $R$ is partitioned into two subsets, $A$ and~$B$. This partition is  known before the game starts.

Let us say that $\alpha$ starts and, so, that  $\beta$ places the final symbol of the string. Player $\alpha$ aims to have the final string in $A$, while player $\beta$ aims for $B$. Zermelo's theorem dictates that either player $\alpha$ has a winning strategy or player $\beta$ has a winning strategy. We refer to~\cite{binmore} for background on Zermelo's theorem and algorithm.

Represent all the posible plays in a binary rooted tree with $2N$ generations (plus the root, the $0$th generation). Each branch of the tree is indexed in an obvious way by a complete string of $L$ (left) and $R$ (right). Label the leaves of the tree corresponding to $A$ with  1 and those corresponding to $B$ with 0. Fill in all the internal nodes (including the root) of the tree backwardly with values 0 and 1  as follows: in the odd numbered generations, place the minimum of the value of the two descendants nodes; and in the even numbered generations, the maximum.

\begin{figure}[h]
\resizebox{13.3cm}{!}{\includegraphics{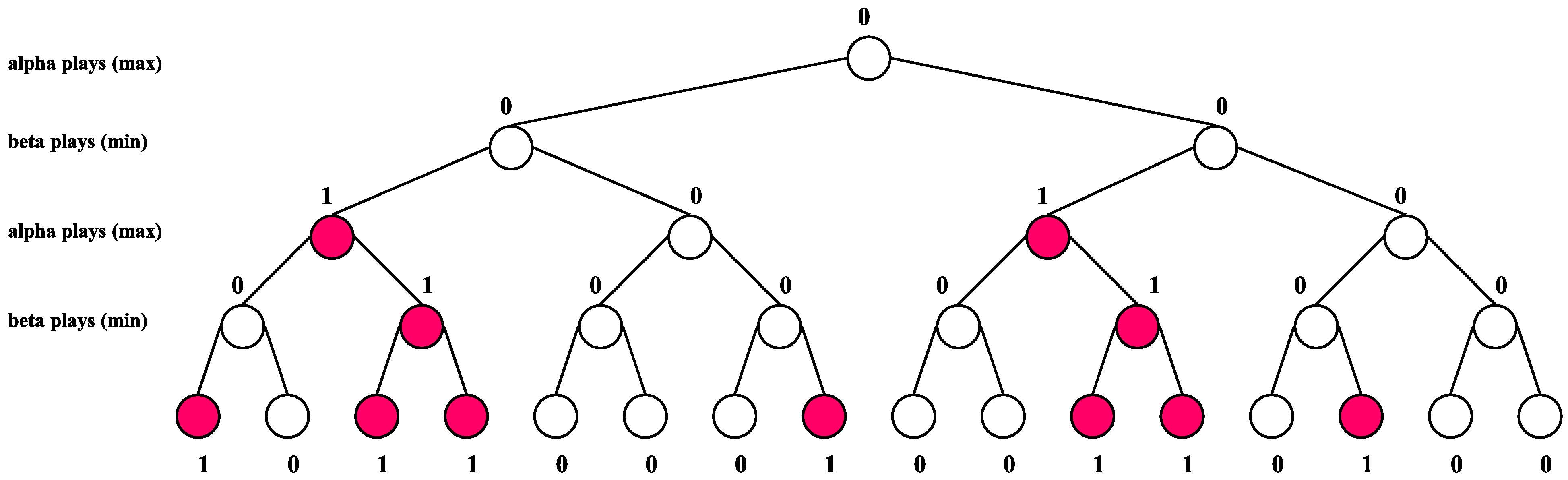}}
\caption{Zermelo's algorithm for $N=2$.}\label{fig:zermelo}
\end{figure}

The value that finally appears at the root is the value $V_N$ of the game from the point of view of $\alpha$, in the sense that  if $V_N=1$, $\alpha$ has a winning strategy, and if $V_N=0$ is $\beta$ who has a winning strategy. Figure \ref{fig:zermelo} would correspond to a winning game for $\beta$.

For each choice of the partition, $A$ and $B$, of the set of the $2^{2N}$ leaves, we obtain in this fashion a well determined value $V_N=V_N(A,B)$.

\smallskip

If we now \textit{randomize} the choice of the partition, $V_N$ becomes  a Bernoulli variable. Fix a probability $p \in (0,1)$ and toss $2^{2N}$ independent coins with success probability $1-p$ to decide for each leaf of the tree whether it is to be included in $A$ (the value 1) or in $B$ (the value 0).

The values of the nodes of the $2N-1$ generation are independent Bernoulli variables with success probability $(1-p)^2$, while the values on the preceding  generation (the $2N-2$ generation) are again independent Bernoulli variables, but now with probability of success~$1-(1-(1-p)^2)^2$. Set
$$
h(p)=(1-(1-p)^2)^2.
$$
 Iterating, we deduce that
$$
\P(V_N=0)=h^{(N)}(p),
$$
where the superscript $N$ on $h$ means that $h$ is composed with itself $N$ times.

The polynomial $h$ increases from $h(0)=0$ to $h(1)=1$, and has a unique fixed point in~$(0,1)$, namely $p^*=1- 1/\varphi=(3-\sqrt{5})/2\approx 0\mbox{.}382$, where $\varphi={(1+\sqrt{5})}/{2}$ denotes the golden section.
\begin{figure}[h]

\centering\resizebox{5cm}{!}{\includegraphics{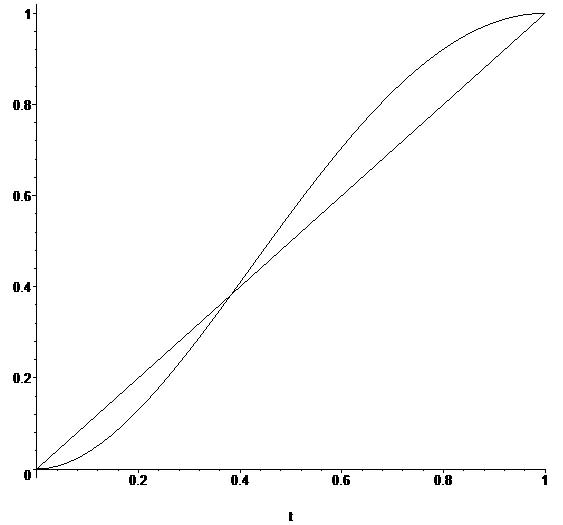}}\qquad \qquad\resizebox{5cm}{!}{\includegraphics{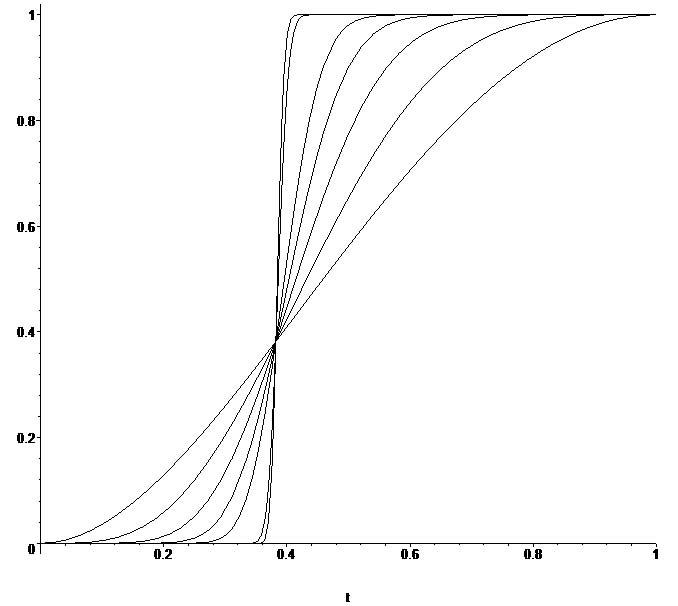}}
\caption{The graphs of the polynomial $h$ of the Zermelo game and of its iterates.}\label{fig:hZermelo}
\end{figure}

As~$N\to\infty$, the iterates $h^{(N)}(p)$ tend to 1 if $p>p^*$; to 0 if $p<p^*$; and to $p^*$ if $p=p^*$.

Thus the random variable $V_N$ tends in distribution to the constant 0 if $p>p^*$; to the constant 1 if $p<p^*$; and to a Bernoulli variable with probability  of success $1-p^*$, if $p=p^*$.

In other terms, for large $N$,
\begin{itemize}\itemsep=0pt
\item player $\beta$ is almost certain to win if $p>p^*$,
\item player $\alpha$ is almost certain to win if $p<p^*$.
\end{itemize}

In terms of quantiles (see Section \ref{notation}),
$$
V_N\xrightarrow{\rm d} Q_X(p^*)\quad\text{as $N\to\infty$},
$$
where $X$ is a Bernoulli variable with probability of success $1-p^*$. %\ojo{ME SALE AL REV\'{E}S}

%\ojo{REVISAR ESTO QUE SIGUE}
If, instead,  player $\beta$ starts  and player $\alpha$ plays the last move, then the critical value $p^*$ is $1/\varphi\approx 0\mbox{.}618$, meaning that $\beta$ is almost certain to win only if $p > 1/\varphi$.

\smallskip

If, further, who plays first is decided by means of a symmetric coin toss, the whole affair becomes equalized and, for large $N$,
\begin{itemize}\itemsep=0pt
\item if $p <1-1/\varphi$, $\alpha$ is almost  certain to win;
\item if $p > 1/\varphi$, $\beta$ is almost certain to win;
\item while, if $1-1/\varphi <p < 1/\varphi$, each player has an equal chance of having a winning strategy.
\end{itemize}

\section{Conservative statistics}\label{section:conservative}

A  (Borel) measurable function $H: \mathbb{R}^n \rightarrow \mathbb{R}$ is said to be a \textbf{conservative statistic} if there exists a function $h\colon [0,1] \to [0,1]$ such that for \textit{any} random variable $X$ the distribution function of $\H(X)$  is given by
$$
F_{\H(X)}=h(F_X)\, ;
$$
that is, if $X_1,\dots, X_n$ are independent copies of $X$,
$$
\P(H(X_1,\dots, X_n)\le t)=h\big(\P(X\le t)\big)\, , \quad \mbox{for each $t \in \mathbb{R}$}\, .
$$
The function $h$ is called the \textit{module} of $H$ and $n$ is termed the \textit{dimension} of $H$. As we will see in a moment, each such module $h$  is a nondecreasing function which satisfies $h(0)=0$ and $h(1)=1$.

Observe that the module $h$ is completely determined by the statistic $H$, as if $U$ is a random variable uniformly distributed in $[0,1]$, then $h$ is the (restriction to $[0,1]$ of the) distribution function of $\H(U)$,
$$
h(t)=\P(\H(U)\le t) \quad\text{for each $t\in[0,1]$.}
$$
In particular, this yields that $h$ is nondecreasing. %For the case of Bernoulli variables, see Lemma \ref{expressionmodulebernoulli}.

\smallskip

The standard projections $H_j(x_1, x_2, \ldots, x_n)=x_j$ for $1 \le j \le n$ are all (trivially) conservative statistics, each of them with the identity as module.

For $n=4$,  the function $M\!m$
\begin{equation}
\label{equation:maxminstatistic}
M\!m(x_1, x_2, x_3, x_4)=\max\big(\min(x_1, x_2), \min(x_3, x_4)\big)
\end{equation}
is a conservative statistic. Observe that
\begin{align*}
&\P\big(\max(\min(X_1,X_2), \min(X_3, X_4))\le t\big)
%=\P\big(\min(X_1,X_2)\le t, \min(X_3, X_4))\le t\big)
%\\
%&\qquad=\big(\P(\min(X_1, X_2)\le t\big)^2=\big(1-\P(\min(X_1, X_2)> t))^2
=\big[1-\big(1-\P(X\le t)\big)^2\big]^2\,,
\end{align*}
so the module of $Mm$ is the polynomial
\begin{equation}\label{module Zermelo game}
h(t)=(1-(1-t)^2)^2.
\end{equation}
This is the statistic pertaining to the randomization of Zermelo's algorithm of Section \ref{randomizingzermelo}.

\subsection{Modules are polynomials}

For the proof of the next lemma,  we shall use Lemma~\ref{lemma:conservative_implies_selector} in Section \ref{section:selectors}, which says that any conservative statistic satisfies the selecting property.

\begin{lemma}
{\textit{The module $h$ of a conservative statistic $H$  of dimension $n$ is  the restriction to $[0,1]$ of a polynomial of degree at most $n$.}}
\end{lemma}
\begin{proof}
  Consider the Boolean cube $\mathbb{B}^n=\{0,1\}^n$. Since $H$ satisfies the selecting property (see Lemma \ref{lemma:conservative_implies_selector}), for each $B\in \mathbb{B}^n$, the value $H(B) \in \{0,1\}$. For $1 \le k\le  n$,
define
\begin{align*}
a_k&= \#\{B\in \mathbb{B}^n : \text{$B$ has $k$ zeros and $H(B)=0$}\},
\end{align*}
 Observe that $a_0=0$ while $a_n=1$, and that  $a_k \le \binom{n}{k}$, for $k=0, 1,  \ldots, n$.

Let $X$ be a Bernoulli variable with probability of success $p$. Since $h(F_X(0))=h(1-p)$, and
$$
\P\big(\H(X)=0\big)=\sum_{k=0}^n a_k\, (1-p)^k \, p^{n-k}\, ,
$$
we conclude that
$$
h(1-p)=\sum_{k=0}^n a_k\, (1-p)^k\,  p^{n-k}\, .
$$
This is true for any $p \in (0,1)$ and therefore
\begin{equation*}
h(t)=\sum_{k=0}^n a_k\, t^k\, (1-t)^{n-k}\, , \quad \text{for any} \quad t \in [0,1]\, .\qedhere
\end{equation*}
%Notice also that, thanks to \eqref{modulus in terms of prob},
%\begin{equation*}
%h(x)=\sum_{k=0}^n a_k\, x^k\, (1-x)^{n-k}.\qedhere
%\end{equation*}
\end{proof}

Observe, from the proof, that the module (of any conservative statistic) may be written~as
\begin{equation}
\label{equation:formulamodules}h(t)=\sum_{k=0}^n a_k \,t^k \,(1-t)^{n-k}=\sum_{k=0}^n \bigg[\frac{a_k}{\binom{n}{k}}\bigg] B^{(n)}_k(t)\, ,
\end{equation}
where the coefficients $a_k$ are \textit{integers} satisfying
\begin{equation}
\label{equation:condition_on_coefs}
a_0=0\, , a_n=1, \quad 0 \le a_k \le \binom{n}{k},  \quad k=0, 1, \ldots, n\, .
\end{equation}
Notice also that, for any module $h$, $h(0)=0$ and $h(1)=1$.

\smallskip
From the proof above,  we deduce  also the following convenient expression for the module~$h$ of a conservative statistic.
\begin{lemma}\label{expressionmodulebernoulli}
For $0 \le t \le 1$, let $J^t$ be a Bernoulli variable with $\P(J^t=0)=t$. Let $H$ be a conservative statistic. Then $\H(J^t)$ is also a Bernoulli variable and
$$
h(t)=\P(\H(J^t)=0)\, .
$$
\end{lemma}

%Two particular examples of modules follow.

\begin{example}{\upshape 1) The module of $Mm$ in \eqref{module Zermelo game} can be written in the form \eqref{equation:formulamodules} as follows:
$$
h(t)=(1-(1-t)^2)^2=4\,t^2\,(1-t)^2+4\,t^3\,(1-t)+t^4,
$$
and so $a_0=a_1=0$, $ a_2=a_3=4$, and $a_4=1$ in this example.

\smallskip
2) For any integer  $n\ge 1$, %the function $H(x_1, \ldots, x_n)=\max{(x_1, \ldots, x_n)}$, or,  more generally,
 the \textit{order statistic} $H_{r:n}$, where $1 \le r \le n$,  orders the  coordinates $(x_1, \ldots, x_n)$ into $x_{i_1} \le x_{i_2} \le \cdots \le x_{i_n}$ and then selects $H_{r:n}(x_1, \ldots, x_n)=x_{i_r}$. In particular, $H_{n:n}(x_1,\dots, x_n)=\max(x_1,\dots, x_n)$ and $H_{1:n}(x_1,\dots, x_n)=\min(x_1,\dots, x_n)$. These $H_{r:n}$ are conservative statistics; the corresponding modules $h_{r:n}$ are (see \cite{David})  the polynomials
\begin{equation}
\label{module order statistic}h_{r:n}(t)=\sum_{j=r}^n \binom{n}{j} \,t^j (1-t)^{n-j}=\sum_{j=r}^n B^{(n)}_j(t)\, ,
\end{equation}
with coefficients $a_j=0$ for $j<r$ and $a_j={n\choose j}$ for $j\ge r$.

In particular, for the maximum, $h_{n:n}(t)=t^n$; and for the minimum, $h_{1:n}(t)=1-(1-t)^n$.}
\end{example}

\begin{remark}
{\upshape The degree of a module $h$ could be smaller than the dimension $n$. For instance the module of $H(x_1, x_2, x_3)=\max(x_1, x_2)$ is $h(t)=t^2=t^2(1-t)+t^3$, so that its coefficients, as in~\eqref{equation:formulamodules},  are $a_0=a_1=0, a_2=a_3=1$. For a projection, say $H(x_1,\dots, x_n)=x_1$, the module is $h(t)=t$, which can be written in the form of \eqref{equation:formulamodules} as
$$
t=\sum_{j=1}^n {{n-1}\choose{j-1}}\, t^j\, (1-t)^{n-j}=\sum_{j=1}^n \frac{j}{n}\, B_{j}^{(n)}(t).
$$}
\end{remark}

\begin{remark}
{\upshape Not all the polynomials of the form  \eqref{equation:formulamodules} and satisfying \eqref{equation:condition_on_coefs} are  modules, simply because they are not necessarily non decreasing in $[0,1]$, as it is shown, for $n=4$,  by the  polynomial $h(t)=2t(1-t)^3+t^4$ with coefficients $a_0=0, a_1=2, a_2=0, a_3=0, a_4=1$.
%\ojo{It would be interesting to determine which polynomials among those satisfying \eqref{equation:formulamodules} and \eqref{equation:condition_on_coefs} are modules (of conservative statistics). See \cite{QRR}.}
}
\end{remark}

\begin{remark}
{\upshape The real polynomials $Q$ of degree $n$ which satisfy $Q(0)=0$, $Q(1)=1$ and $0<Q(x)<1$ for any $x \in (0,1)$ are precisely those polynomials which \textit{may be expressed} as
$$
Q(x)=\sum_{j=0}^m \beta_j\, B^{(m)}_j(x)\,,
$$
with $\beta_0=0$, $\beta_m=1$ and $0 \le \beta_j \le 1$ for $j=0,1,\dots m$. The Bernstein degree $m$ is usually larger than $n$. See {\upshape \cite{QRR}}.}
\end{remark}

\subsection{Sperner statistics}\label{section:sperner}

Sperner statistics, which we are about to introduce, are precisely, as we will show later on (Theorem \ref{theor:selectors equal sperner}),
the \textit{continuous} conservative statistics. %Zermelo statistics are a particular case of Sperner statistics.

\smallskip

For each subset $A$ of $\{1,\dots,n\}$ we introduce the function
$\min\nolimits_{A}\colon \mathbb{R}^n \to \mathbb{R}$ given~by
$$
\min_A(x_1, x_2, \ldots, x_n)=\min\{x_i, i \in A\}\, ,
$$
which gives the minimum of the values of the coordinates corresponding to the index subset~$A$. Correspondingly,
$$
\max_A(x_1, x_2, \ldots, x_n)=\max\{x_i, i \in A\}\, .
$$

\smallskip

To each Sperner family $\mathcal{S}=\{A_1, \ldots, A_k\}$ of subsets of $\{1, 2, \ldots n\}$ we associate a \textbf{Sperner statistic} $H_{\mathcal{S}}$ in $\mathbb{R}^n$  given by
\begin{equation}\label{Sperner statistic}
H_{\mathcal{S}}=\max\big(\min_{A_1}, \min_{A_2}, \ldots, \min_{A_k}\big)\, .
\end{equation}

The statistic $H_{\mathcal{S}}$ is a projection (onto a certain coordinate) if and only if $\mathcal{S}$ consist of just one singleton.

The Sperner statistic corresponding to the family consisting of all the subsets of $\{1,\dots,n\}$ of size $r$ is precisely the order statistic $H_{n-r+1:n}$. The Sperner statistic of any disjoint family is termed a \textbf{Zermelo statistic}.

\begin{lemma}
Any Sperner statistic $H_{\mathcal{S}}$ is conservative, and actually, its module $h_{\mathcal{S}}$  %of $H_{\mathcal{S}}$
is %thus
given by the polynomial
\begin{equation}\label{formulamodulesperner}
h_{\mathcal{S}}(t)=1-\sum_{i}(1-t)^{|A_i|}+\sum_{i<j} (1-t)^{|A_i \cup A_j|}-\cdots
\end{equation}
\end{lemma}
\begin{proof}
Let $|\mathcal{S}|=k$ and let $X_1, \ldots, X_n$ be independent copies of a random variable~$X$. Using inclusion/exclusion, write now, for any $t \in [0,1]$,
\begin{equation*}
\begin{aligned}
h_\mathcal{S}(t)&=       \P\big(H_{\mathcal{S}}\big(X_1,  \ldots, X_n\big)\le t\big)=
 \P\Big(\bigcap_{i=1}^k\big\{\min_{A_i}\big(X_1,  \ldots, X_n\big) \le t\big\}\Big)
 \\&=1-\P\Big(\bigcup_{i=1}^k\big\{\min_{A_i}\big(X_1,  \ldots, X_n\big) > t\big\}\Big)
 \\&=1-\sum_{i=1}^k\P\Big(\big\{\min_{A_i}\big(X_1,  \ldots, X_n\big)> t\big\}\Big)+\sum_{i<j} \P\Big(\big\{\min_{A_i\cup A_j}\big(X_1,  \ldots, X_n\big)> t\big\}\Big)-\cdots
 \\&=1-\sum_{i=1}^k(1-F_X(t))^{|A_i|}+\sum_{i<j} (1-F_X(t))^{|A_i \cup A_j|}-\cdots
\end{aligned}
\end{equation*}
This gives \eqref{formulamodulesperner}. Observe also that
\begin{equation*}
1-h_{\mathcal{S}}(1-t)=\sum_{i} t^{|A_i|}-\sum_{i<j} t^{|A_i \cup A_j|}-\ldots\,.\qedhere
\end{equation*}
\end{proof}
For the particular case of a Zermelo statistic, $\mathcal{S}$ is a disjoint family and  the expression of~$h_{\mathcal{S}}$ simplifies to
\begin{equation}
\label{module Zermelo statistic}
h_{\mathcal{S}}(t)=\prod_{i} \big(1-(1-t)^{|A_i|}\big)\, .
\end{equation}
For the module of an order statistic, see \eqref{module order statistic}.

%\begin{remark}
%{\upshape The operator $H_\mathcal{S}$ given in \eqref{Sperner statistic} can be defined for an arbitrary family $\mathcal{S}$ of subsets of $\{1,\dots,n\}$. But in this case we can reduce $\mathcal{S}$ to Sperner family by removing redundant sets.
%
%On the other hand, if $\mathcal{S}=\{A_1,\dots, A_k\}$ is a Sperner family,
%}
%\end{remark}

\begin{remark}[On Sperner and arbitrary families]
{{\upshape The operator \eqref{Sperner statistic} can be defined for arbitrary families $\mathcal{A}=\{A_1,\dots, A_k\}$ of subsets of $\{1,\dots,n\}$. But observe that
$$
H_\mathcal{A}=\max\big(\min_{A_1}, \dots, \min_{A_k}\big)
$$
would coincide with the Sperner statistic $H_\mathcal{S}$, where $\mathcal{S}$ is the Sperner family obtained by retaining only the minimal (with respect to inclusion) elements of the family $\mathcal{A}$. Clearly, $h_\mathcal{A}=h_\mathcal{S}$.

Also, given a Sperner family $\mathcal{S}$, one could define the associated upset
$$
\mathcal{U}=\{U\subset \{1,\dots,n\}: U\supseteq A_i\text{ for some $A_i\in\mathcal{S}$}\}.
$$
Again, $H_\mathcal{S}\equiv H_{\mathcal{U}}$ and $h_\mathcal{S}\equiv h_{\mathcal{U}}$. See also the remarks preceding Theorem~\ref{equation:monotone_representation}.}}
\end{remark}

The case of Sperner statistics with the identity as module is a bit special.
\begin{lemma}\label{lemma:module identity equal projection}
The  module $h_{\mathcal{S}}$ of a Sperner statistic is  the identity if and only if $\mathcal{S}$ consists of  just one singleton.  In this case,  as we have seen, $H_{\mathcal{S}}$ is a projection.
\end{lemma}
\begin{proof}
The converse part is obvious. %Assume that $h_{\mathcal{S}}$ is  the identity.
Write
$$
h_{\mathcal{S}}(t)=1-\sum_{i}(1-t)^{|A_i|}+\sum_{i<j} (1-t)^{|A_i \cup A_j|}-\cdots\,,
$$
and observe that if $h_{\mathcal{S}}$ is the identity, then for any $t \in [0,1]$,
\begin{equation}
\label{para id}
t=\sum_{i}t^{|A_i|}-\sum_{i<j} t^{|A_i \cup A_j|}+\cdots
\end{equation}
Assume that $|A_1| \le |A_2| \le \cdots \le |A_k|$. Observe that $\mathcal{S}$ must contain (at least) one singleton, so that $|A_1|=1$. We must show that $k=1$.

\smallskip

Assume that $k \ge 2$. Notice that $|A_2|=1$ is impossible (as the coefficients of the linear terms in \eqref{para id} would not match). So $|A_2|\ge 2$.

As no $A_i$ is contained in any other $A_j$, we have that
%$|A_i\cup A_j| \ge |A_1|+1$, if $i \neq j$. Consequently, $|A_1|=1$ and  $|A_2| \ge 2$.
%
%Then
$$
\begin{aligned}
|A_1 \cup A_j|\ge |A_j|+1 \ge |A_2|+1, &\quad \text{if} \quad j \ge 2\, , \\
|A_i \cup A_j| \ge |A_i|+1\ge |A_2|+1, &\quad \text{if} \quad 1 <i <j\,.
\end{aligned}
$$
Thus if, say, $|A_2|=|A_3|=\cdots =|A_l|$ and $|A_l|< |A_{l+1}|$ or $l=k$, then
$$
t=t+(l-1)t^{|A_2|}+\text{higher order terms}\, , \quad \mbox{for any $t \in [0,1]$}\, ,
$$
which is impossible.\end{proof}

\begin{remark}[Isomorphic statistics] \label{remark:same module does not imply isomorphic}{\upshape
We say that two conservatives statistics $H$ and $G$ in~$\mathbb{R}^n$ are  \textit{isomorphic} if there exists a permutation $\sigma$ of the index set $\N_n$ such that
$$
G(x_1, x_2, \ldots, x_n)=H(x_{\sigma(1)},x_{\sigma(2)}, \ldots, x_{\sigma(n)})\, .
$$
Obviously, isomorphic conservative statistics have the same module; the converse does not hold in general.
For instance, the statistics associated to the Sperner families $\mathcal{F}=\{\{1, 2\}, \{3, 4\}\}$ and $\mathcal{G}=\{\{1, 2\}, \{1, 3\}, \{2, 3, 4\}\}$ have the same module, namely $h(t)=1-2(1-t)^2+(1-t)^4$. Notice that, by Lemma {\rm \ref{lemma:module identity equal projection}}, two Sperner statistics with the identity as module are isomorphic. }
\end{remark}

\subsubsection{Some properties of modules of Sperner statistics}

Next, we collect a few useful observations about the modules of Sperner statistics.

\begin{lemma}\label{derivadas en 0 y 1}Let $\mathcal{S}=\{A_1, \ldots, A_k\}$ be a Sperner family  $($not consisting of just one singleton$)$. Then
$$
    h_{\mathcal{S}}^{\prime}(0)=\Big|\displaystyle\bigcap_{j=1}^k A_j\Big|\, ,\quad\text{and}\quad
    h_{\mathcal{S}}^{\prime}(1)=\mbox{number of singletons among the $A_j$}\,.
    $$
          Notice also  that if $h_{\mathcal{S}}^{\prime}(1)>0$ then $h_{\mathcal{S}}^{\prime}(0)=0$, so that always $h_{\mathcal{S}}^{\prime}(0)\cdot h_{\mathcal{S}}^{\prime}(1)=0$.
\end{lemma}
\begin{proof}
At $t=0$ we have, by \eqref{formulamodulesperner}, that
$$
h_{\mathcal{S}}^{\prime}(0)=\sum_{i}|A_i|-\sum_{i<j} |A_i \cup A_j| + \cdots\, ,
$$
which by inclusion/exclusion may be written, using that $\sum_{k=0}^n (-1)^k {n\choose k}=0$, as
$$
h_{\mathcal{S}}^{\prime}(0)
=\Big|\bigcap_{j=1}^k A_j\Big|\,.
$$

At $t=1$, using again  \eqref{formulamodulesperner}, we have
$$
h_{\mathcal{S}}^{\prime}(1)=\sum_{i; |A_i|=1} |A_i|=\mbox{number of singletons among the $A_i$}\, .
$$

If $h_{\mathcal{S}}^{\prime}(1)>0$, then $\mathcal{S}$ contains (at least) one singleton, which has no intersection with any of the other $A_j\in \mathcal{S}$, and so $h_{\mathcal{S}}^{\prime}(0)=0$.
\end{proof}

\begin{lemma}\label{lemma:bounds zero one}Let $\mathcal{S}$ be a Sperner family. For $t\in(0,1)$,
$$
0< h_\mathcal{S}(t)<1.
$$
\end{lemma}
\begin{proof} Just observe that
$\min\le H_{\mathcal{S}}\le \max$, so that, for $t\in(0,1)$,
$$
0<t^n\le h_\mathcal{S}(t)\le 1-(1-t)^n<1.
$$

These bounds can be easily improved. Let $M=M(\mathcal{S})$ be the size of the smallest member of $\mathcal{S}$. Say $|A_1|=M$. Then, $h_{\mathcal{S}}(t)\le 1-(1-t)^{M}<1$, as $\min\nolimits_{A_1}\le H_{\mathcal{S}}$. On the other hand, let $b=b(\mathcal{S})$ the smallest size of a set $B\subset \N_n$ (if any) which intersects every $A_j\in \mathcal{S}$. Then, for every $t\in (0,1)$,
$
0<t^b \le h_{\mathcal{S}}(t)$, as $H_{\mathcal{S}}\le \max_B$. %\ojo{QUIZ\'{A}S QUITAR, NO SE USA?}
\end{proof}

\begin{lemma}\label{lemma:basis for recursion}
Let $\mathcal{S}$ be a Sperner family $($not consisting of just one singleton$)$. Then:
\begin{itemize}
\item[\rm a)] If\/ $h_\mathcal{S}^{\prime}(1) >0$, then $h(t) <t$, for any $t \in (0,1)$.
\item[\rm b)] If\/ $h_\mathcal{S}^{\prime}(0) >0$, then $h(t) >t$, for any $t \in (0,1)$.
\end{itemize}
\end{lemma}

\begin{proof}
a) If $h_\mathcal{S}^{\prime}(1) >0$, then, by Lemma \ref{derivadas en 0 y 1}, $\mathcal{S}$ contains a singleton, say $A_1=\{a\}$. Let $\mathcal{T}=\mathcal{S}\setminus A_1\neq \emptyset$. Each member of $\mathcal{T}$ has empty intersection with $A_1$. We claim that
$$
h_\mathcal{S}(t)=t\, h_{\mathcal{T}}(t) \quad\text{for any $t \in (0,1)$.}
$$
Consider $U_1, \ldots, U_n$ independent uniform variables in $[0,1]$. Then
$$\begin{aligned}
h_\mathcal{S}(t)&=h_\mathcal{S}(\P(U\le t))=\P\Big(\bigcap_{j=1}^k \big\{\min_{A_j}(U_1, \ldots, U_n)\le t\big\}\Big)
\\&=\P\Big(\{U_a\le t\}\ {\textstyle{\bigcap}}\ \bigcap_{j=2}^k \big\{\min_{A_j}(U_1, \ldots, U_n)\le t\big\}\Big)
=t\, \P\Big(\bigcap_{j=2}^k \big\{\min_{A_j}(U_1, \ldots, U_n)\le t\big\}\Big)
\\
&=t\, \P(H_\mathcal{T}(U_1,\dots, U_n)\le t)=t\, h_\mathcal{T}(\P(U\le t))=t\, h_\mathcal{T}(t).
\end{aligned} $$
From Lemma \ref{lemma:bounds zero one} and the fact that the family $\mathcal{T}$ is not empty, we conclude that $h_\mathcal{S}(t)<t$.

\smallskip

b) If $h_\mathcal{S}^{\prime}(0) >0$, then $\bigcap_{j=1}^k A_j\neq \emptyset$, say $1 \in \bigcap_{j=1}^k A_j$. Now the singleton $\{1\}$ is not a member of $\mathcal{S}$. Define a new Sperner family $\mathcal{T}=\{B_1, B_2, \ldots, B_k\}$, where, for $1 \le j \le k$, we set $B_j=A_j\setminus \{1\}$, and observe that each $B_j$ is not empty.

Again, if $U_1, \ldots, U_n$ are independent uniform variables, we have, for  each $t \in [0,1]$
$$\begin{aligned}
&\P\Big(\bigcap_{j=1}^k \big\{\min_{A_j}(U_1, \ldots, U_n)\le t\big\}\Big)
\\
&\ =\P\Big(\bigcap_{j=1}^k \big\{\min_{A_j}(U_1, \ldots, U_n)\le t\big\}\bigcap \{U_1 \le t\}\Big)
+
 \P\Big(\bigcap_{j=1}^k \big\{\min_{A_j}(U_1, \ldots, U_n)\le t\big\}\bigcap \{U_1 > t\}\Big)
 \\
  &\ =\P\big( U_1 \le t\big)+
 \P\Big(\bigcap_{j=1}^k \big\{\min_{B_j}(U_1, \ldots, U_n)\le t\big\}\bigcap \{U_1 > t\}\Big)
 \\
 &\ =\P\big( U_1 \le t\big)+
 \P\Big(\bigcap_{j=1}^k \big\{\min_{B_j}(U_1, \ldots, U_n)\le t\big\}\Big)\,\P\big(U_1 > t\big)\,.
\end{aligned} $$
Therefore, for each $t \in [0,1]$,
$$
h_{\mathcal{S}}(t)=t+(1-t)\,h_{\mathcal{T}}(t)\, .
$$
and consequently, because of Lemma \ref{lemma:bounds zero one}, we have
$
h_{\mathcal{S}}(t)>t
$, for $t \in (0,1)$.
\end{proof}

We now describe a recursive construction of $h_{\mathcal{S}}$ based upon expressing $h_{\mathcal{S}}$ in terms of modules associated to smaller families. This construction is somehow implicit in  the proof of Lemma \ref{lemma:basis for recursion}.

 For $1 \le r \le n$, we define $\mathcal{S}\setminus r$ as follows: from  the family $\{A_1\setminus \{r\}, \ldots, A_k\setminus \{r\}\}$ remove successively any set which is superset of any other set in the (remaining) family. The resulting family is a Sperner family of $\{1,\dots,n\} \setminus \{r\}$ unless $\mathcal{S}$ contains the singleton~$\{r\}$;  in this case we end up with $\mathcal{S}\setminus r=\emptyset$, and we conventionally agree that $h_{\mathcal{S}\setminus r} \equiv 0$.

We also define, for $1 \le r \le n$, the family  $\mathcal{S}\vdash r=\{A_j, 1 \le j \le k, r \notin A_j\}$. This family is a Sperner family (of $\{1,\dots,n\} \setminus \{r\}$) unless $r \in \bigcap_{j=1}^k A_j$; in this case we end up with $\mathcal{S}\vdash r=\emptyset$ and we conventionally agree that $h_{\mathcal{S}\vdash r} \equiv 1$.

With these two operations and the corresponding conventions we may state:
\begin{lemma}\label{lemma:recursion} For each $t \in [0,1]$ and $1 \le r \le n$,
$$
h_{\mathcal{S}}(t)=t \,h_{\mathcal{S}\vdash r}(t)+(1-t) \,h_{\mathcal{S}\setminus r}(t)\, .
$$
\end{lemma}
\begin{proof}
Write $h_{\mathcal{S}}(t)=\P\big(H_{\mathcal{S}}(U_1, \ldots, U_n)\le t\big)$, where the $U_j$ are uniform and independent random variables, and condition on the partition $\{U_r\le t, U_r >t\}$.
\end{proof}

\begin{remark}[Stochastic Logic]{\upshape
In Stochastic Logic (see \cite{QuianRiedel} and \cite{QuianRiedeletal}), for a general Boolean function $J$ and for probabilities $0\le p_r \le 1$ for $1 \le r \le n$,  one considers independent Bernoulli variables $B_{p_r}$, $ 1 \le r \le n$, with $\P(B_{p_r}=0)=p_r$ and observes that the variable $J(B_{p_1}, \ldots, B_{p_n})$ is Bernoulli with a probability of attaining 0 given by a multilinear polynomial on $p_1, \ldots, p_n$ of degree 1 in each variable.

For a Sperner statistic $H_{\mathcal{S}}$ of $\mathcal{S}=\{A_1, \ldots, A_k\}$ one has that
$$
\P\big(H_{\mathcal{S}}\big(B_{p_1}, \ldots, B_{p_n}\big)=0\big)=1-\sum_{r} \prod_{i \in A_r}(1-p_i)+\sum_{r <s} \,\prod_{i \in A_r\cup A_s}(1-p_i)-\cdots
$$
Compare with equation \eqref{formulamodulesperner} and Lemma \ref{expressionmodulebernoulli}.}
\end{remark}

\section{Selectors}\label{section:selectors}
We shall show now that selectors are the \textit{continuous} conservative statistics (see Theorem~\ref{theor:continuous conservative equal selector}). Further, in Section~\ref{subsection:selectors and Sperner}, we will show that selectors are exactly the Sperner statistics we have just introduced. The key observation for this latter result is that selectors are determined by their restriction to the Boolean cube $\mathbb{B}^n$, and that they are monotone in~$\mathbb{R}^n$ (and in $\mathbb{B}^n$).
Both characterizations of selectors appear summarized in Theorem \ref{theor:continuous conservative equal selectors equal sperner}.

%The main result of this section is that continuous conservative statistics and selectors coincide, Theorem \ref{theor:continuous conservative equal selector}, and further that both are precisely the Sperner statistics, %Theorem~\ref{theor:selectors equal sperner}. \ojo{UNA FRASE SOBRE MONOTON\'{I}A Y BOOLEANO, VER COMIENZO DE SEC \ref{subsection:selectors and Sperner}.}

\subsection{Selectors and conservative statistics}\label{section:poly g}

\begin{lemma}\label{lemma:conservative_implies_selector}
Any conservative statistic satisfies the selecting property.
\end{lemma}
\begin{proof}
Let $H$ be conservative and let $(x_1, \ldots, x_n)$ be any point in $\mathbb{R}^n$. Let $X$ be a random variable such that $\P(X=x_j)=1/n$. Observe that the coordinates $x_j$ are not necessarily all distinct. The distribution function of $X$ has jumps exactly at the $x_j$, and, therefore, the distribution function $h(F_X)$ has jumps at most at the $x_j$, as $h$ is a nondecreasing polynomial. We conclude that the random variable $H(X_1, \ldots, X_n)$ takes  values only on $\{x_1, \ldots, x_n\}$.
\end{proof}\endproof

\textit{There are statistics satisfying the selecting property  which are not conservative statistics.} Define $H$ in $\mathbb{R}^2$ by
$$
H(x,y)=\begin{cases}
x, & \text{if} \quad x \le 0\, ,\\
y, & \text{if} \quad x >0\, .
\end{cases}
$$
Assume that $H$ is conservative with module $h$. Fix $p \in (0,1)$ and let $X$ be the variable $\P(X=-1)=p$, $P(X=1)=1-p$, and let $Y$ be an independent copy of $X$. Now, $H(X,Y)$ takes the value $-1$ with probability $1-(1-p)^2$, and thus $h(p)=1-(1-p)^2$, for every $p \in (0,1)$. Again fix $p \in (0,1)$ and let $X$ be the variable $\P(X=1)=p$, $\P(X=2)=1-p$, and let $Y$ be an independent copy of $X$. Now, $H(X,Y)=Y$, and therefore $h(p)=p$, for every $p \in (0,1)$. This contradiction shows that $H$ is not conservative.

\begin{lemma}\label{lemma:selectorcontinuousimpliesconservative}
Any  selector is a conservative statistic.
\end{lemma}
\begin{proof}
Recall that a selector is a \textit{continuous} function. Let $\vec{\boldsymbol{\varepsilon}}=(\varepsilon_1, \ldots, \varepsilon_n)$ be any list of the symbols $\pm 1$ and let $t \in \mathbb{R}$. Define
$$
\Omega_t(\vec{\boldsymbol{\varepsilon}})=\big\{(x_1, \ldots, x_n): \text{$x_j >t$ if $\varepsilon_j=+1$; and $x_j \le t$ if $\varepsilon_j=-1$, for $j=1, 2, \ldots, n$}\}.
$$
For each $t$, the collection of the $2^n$ subsets of the form $\Omega_t(\vec{\boldsymbol{\varepsilon}})$ constitutes  a partition of $\mathbb{R}^n$.

For given $t$ and given $\vec{\boldsymbol{\varepsilon}}$, we have that $H\big(\Omega_t(\vec{\boldsymbol{\varepsilon}})\big) \subset (t, +\infty)$ or $H\big(\Omega_t(\vec{\boldsymbol{\varepsilon}})\big) \subset (-\infty, t]$. (Since~$H$ is a selector, this is clearly so  if all the $\varepsilon_j$ are equal.) Assume that this is not the case, and that for $\mathbf{x}=(x_1, \ldots, x_n), \mathbf{y}=(y_1, \ldots, y_n) \in \Omega_t(\vec{\boldsymbol{\varepsilon}})$ and $H(\mathbf{x}) >t$ while $H(\mathbf{y}) \le t$. Using both that $H$ is continuous and satisfies the selecting property, we may perturb both~$\mathbf{x}$ and~$\mathbf{y}$ and assume that~$\mathbf{x}$ and~$ \mathbf{y}$ are in the topological interior of $\Omega_t(\vec{\boldsymbol{\varepsilon}})$ and that $H(\mathbf{x}) >t$ while $H(\mathbf{y}) < t$ (strict inequality now). Continuity of  $H$ would give the existence of $\mathbf{z}$ in  the  interior of $\Omega_t(\vec{\boldsymbol{\varepsilon}})$, with $H(\mathbf{z})=t$, but this is impossible since $H$ is a selector.

Reasoning as above, for  $\vec{\boldsymbol{\varepsilon}}$ fixed, if for a single value of $t$ we have $H\big(\Omega_t(\vec{\boldsymbol{\varepsilon}})\big) \subset (-\infty, t]$, then  this is the case for every $t$. Define $a_k$ as the number of sets $\Omega_t(\vec{\boldsymbol{\varepsilon}})$ where $\vec{\boldsymbol{\varepsilon}}$ has exactly~$k$ coordinates $-1$ and $H\big(\Omega_t(\vec{\boldsymbol{\varepsilon}})\big) \subset (-\infty, t]$, for $k=0,1, \ldots, n$. Observe that $a_k$ does not depend on $t$.

Finally, for any random variable $X$ we have
$$
\P\big(\mathbf{H}(X)\le t\big)=\sum_{k=0}^n a_k F_X(t)^k (1-F_X(t))^{n-k}=h(F_X(t))
$$
where $h$ is the polynomial $h(x)=\sum_{k=0}^n a_k x^k (1-x)^{n-k}$. We conclude that $H$ is  a conservative statistic.
\end{proof}

\textit{There are conservative statistics which are not continuous.} Take a disk $D$ in contained in $\{x>y\} \subset\mathbb{R}^2$, and let $\hat{D}$ be its symmetric image with respect to the line $y=x$. Denote $C=D\cup \hat{D}$ and define $H(x,y)$ by
$$
H(x,y)=\begin{cases}
y, & \text{if} \quad (x,y) \in C,\\
x, & \text{otherwise.}
\end{cases}
$$
This function $H$ satisfies the selecting property and it is not continuous. Let $X, Y$ be independent and identically distributed. Now, for each $t \in \mathbb{R}$ and by symmetry,
$$\begin{aligned}
\P(H(X,Y)\le t)&=\P(X \le t; (X,Y) \notin C)+\P(Y \le t; (X,Y) \in C)\\&=\P(X \le t; (X,Y) \notin C)+\P(X \le t; (X,Y) \in C)=P(X\le t)\,,
\end{aligned}
$$
so that $H$ is conservative with module $h(x)=x$.

\begin{theorem}\label{theor:continuous conservative equal selector}
Any continuous conservative statistic is a selector, and conversely.
\end{theorem}
\begin{proof}
This is a consequence of Lemmas \ref{lemma:conservative_implies_selector} and \ref{lemma:selectorcontinuousimpliesconservative}.
\end{proof}

Among the selectors, the order statistics may be characterized as follows.

\begin{lemma}
A {\upshape symmetric} selector $H$ is an order statistic, and conversely..
\end{lemma}
By \textit{symmetric} we mean that the value of $H$ is unchanged if the coordinates are reordered.

\begin{proof}
Since $H$ is symmetric, $H$ is determined by its restriction to the set $\{(x_1, \ldots, x_n): x_1 \le \cdots \le x_n\}$, and, in fact, since $H$ is continuous, $H$ is determined by its restriction to $\{(x_1, \ldots, x_n): x_1 < \cdots < x_n\}$. But a selector on $\{(x_1, \ldots, x_n): x_1 < \cdots < x_n\}$ must select the same coordinate for all points.
\end{proof}

%The statistic of dimension $3$ given by $H(x_1, x_2, x_3)=\max(x_1,x_2)$ is conservative with module $h(x)=x^2$, but it is not symmetric, and not an order statistic.

\begin{remark}
{\upshape The only $C^1$ selectors $H$ are the projections: constantly selecting a fixed  coordinate. This is so because at any point $(x_1, x_2, \ldots, x_n)$ with no repeated coordinates,  the gradient of $H$ has to be one of the vectors in the standard basis.}
\end{remark}

\subsection{Selectors and Sperner statistics}\label{subsection:selectors and Sperner}

In this section we show that every  selector is a Sperner statistic; since Sperner statistics are obviously  selectors, the two notions coincide.

We start by pointing out two further  properties of selectors. Once we have shown that selectors are Sperner statistics, those two properties f selectors will be obvious, but they are instrumental (in our approach) for showing that the two notions coincide.

A function $G: \mathbb{R}^n \to \mathbb{R}$ is called \textit{monotone} if
$$
G(x_1, x_2, \ldots, x_n) \le G(y_1, y_2, \ldots, y_n)
$$
whenever $x_i \le y_i$, for $i=1, 2, \ldots, n$.

\begin{lemma}\label{lemma:continuous_selector_monotone}
Any selector $H$ is monotone.
\end{lemma}
\begin{proof}
It is enough to show that for any given $(x_2, x_3, \ldots, x_n)$ the function
$$
x \in \mathbb{R} \mapsto u=g(x)=H(x, x_2, \ldots, x_n)\, ,
$$
is increasing. Since $H$ is a selector, the graph of $g$ in the  $(x,u)$ plane is contained in the union of the horizontal lines $\{u=x_2\}, \{u=x_3\}, \ldots, \{u=x_n$\} and the line $\{u=x\}$, and since $g$ is a continuous function, we conclude, as desired,  that $g$ is non-decreasing. In fact~$g$ has to be one of the following five types of functions: $g(x)=x_k$, for some $x_k$ and for all~$x \in \mathbb{R}$; or $g(x)=x$ for all~$x \in \mathbb{R}$; or $g(x)=\max(x,x_j)$, for some $x_j$ and for all~$x \in \mathbb{R}$; or $g(x)=\min(x,x_k)$, for some $x_k$ and for all $x \in \mathbb{R}$; or, finally, $g(x)=\max(\min(x,x_j),x_k)$ for some $x_k < x_j$ and for all $x \in \mathbb{R}$.
\end{proof}

There are \textit{monotone functions which satisfy the selecting property and  which are not continuous} and, even further,  which are \textit{not conservative statistics}. For instance, let \mbox{$H(x,y)=y$}  if $y \ge |x|$, and $H(x,y)=x$ otherwise. Clearly, $H$ is monotone and satisfies the selecting property. Let us assume that $H$ is a conservative statistic with module $h$. Consider $X,Y$  independent Bernoulli variables with parameter $1-p \in (0,1)$; then $H(X,Y)$ is Bernoulli with parameter $1- p^2$, and consequently, $h(t)=t^2$, for any $t \in (0,1)$ (see Lemma \ref{expressionmodulebernoulli}). Consider now $X,Y$ independent and uniformly distributed in $\{-1, 0, +1\}$; now $\P(H(X,Y)\le -1)=2/9$ and $\P(X\le -1)=1/3$ which would imply $h(1/3)=2/9$, instead of $h(1/3)=1/9$.

\begin{lemma}\label{lemma:continuous_selectors_homogeneity}
Selectors $H$ are positively homogeneous of degree $1${\upshape:} for $\lambda >0$ and any $x_1, \ldots, x_n \in \mathbb{R}$,
$$
H(\lambda x_1, \ldots, \lambda x_n)=\lambda \,H(x_1, \ldots, x_n)\, .
$$
More generally, if $f$ is an increasing homeomorphism of\/ $\mathbb{R}$, then
$$
H\big(f(x_1), \ldots, f(x_n)\big)=f\big(H(x_1, \dots, x_n)\big)\, .
$$
\end{lemma}
\begin{proof}
This follows immediately from the fact that for any permutation $\sigma$ of $\{1,\dots,n\}$, a selector restricted to
$$
\big\{(x_1, x_2, \ldots, x_n) \in \mathbb{R}^n: x_{\sigma(1)} < x_{\sigma(2)}< \cdots < x_{\sigma(n)}\big\}
$$
is a projection.
\end{proof}

The following lemma shows that it is enough to consider selectors as Boolean functions.

\begin{lemma}\label{lemma:continuous_selectors_unit_cube}
If two selectors in $\mathbb{R}^n$ coincide on $\mathbb{B}^n=\{0,1\}^n$ then they coincide everywhere.
\end{lemma}
\begin{proof}
Let $H$ and $J$ be two selectors in $\mathbb{R}^n$. Assume that $H$ and $J$ coincide on $[0,1]^n$. For  any $a >0$, using $f(x)=2ax-a$ in Lemma \ref{lemma:continuous_selectors_homogeneity},  we see that $H,J$ coincide on
$[-a,a]^n$, and consequently $H\equiv J$.

Next, we show that if $H,J$ coincide on $\{0,1\}^n$, then they coincide on $[0,1]^n$. To verify this we now consider selectors defined only on $[0,1]^n$ and not in the whole of $\mathbb{R}^n$. Those selectors satisfy Lemmas \ref{lemma:continuous_selector_monotone} and \ref{lemma:continuous_selectors_homogeneity}.

We will prove by induction that a selector $\widetilde{H}$ defined on $[0,1]^n$ is determined by its values on $\{0,1\}^n$. Assume that this is true for dimension $n$. Let $\widetilde{H}$ be a selector defined on~$[0,1]^{n+1}$. Consider the restriction $G$ of $\widetilde{H}$  to the face of the boundary of $[0,1]^{n+1}$ given by $x_{n+1}\equiv1$,~i.e.,
$$(x_1, x_2, \ldots, x_n) \mapsto G(x_1, x_2, \ldots, x_n) =\widetilde{H}(x_1, x_2, \ldots, x_n, 1)$$

There are two possibilities. First, if for some point $(\bar{x}_1, \ldots, \bar{x}_n) \in (0,1)^n$ we have $G(\bar{x}_1, \ldots, \bar{x}_n)=1$, then since $\widetilde{H}$ is a selector, $\widetilde{H}(x_1, \ldots, x_n, 1)=1$ for each $({x_1}, \ldots, {x_n}) \in (0,1)^n$. Consequently,
\begin{equation}\label{equation:widetilde_on_corners}
\widetilde{H}(\epsilon_1, \ldots, \epsilon_n,1)=1 \quad \text{for any} \quad (\epsilon_1, \ldots, \epsilon_n) \in \{0,1\}^n\,.
\end{equation}
Conversely, if \eqref{equation:widetilde_on_corners} happens, then just by monotonicity
\begin{equation}\label{equation:widetilde_inside}
\widetilde{H}(x_1, \ldots, x_n,1)=1, \quad \text{for any} \quad (x_1, \ldots, x_n) \in [0,1]^n\,.
\end{equation}

The other possibility is that $G(x_1,  \ldots, x_n) \neq 1$, for each $({x_1}, \ldots, {x_n}) \in (0,1)^n$. Then $G$ is a selector in $[0,1]^n$, and by induction $G$ is determined by its values at the corners $\{0,1\}^n$.

Arguing similarly with the other faces of $\partial [0,1]^{n+1}$ we conclude that the restriction of $\widetilde{H}$ to $\partial [0,1]^{n+1}$ is determined by its values on $\{0,1\}^{n+1}$. By homogeneity (Lemma \ref{lemma:continuous_selectors_homogeneity}) we conclude that $\widetilde{H}$ is determined in the whole of $[0,1]^{n+1}$ by its values on $\{0,1\}^{n+1}$, as desired.

The case $n=2$, to start the induction argument, is quite direct. Let $H$ be a selector on $[0,1]^2$. There are four cases to consider, given by the values of $H$ on the corners $a=(0,1)$ and on $b=(1,0)$. If $H(a)=0,H(b)=0$, then $H(x,0)=0$, $H(0,y)=0$, $H(x,1)=x$, $H(1,y)=y$, for any $x,y \in [0,1]$. By homogeneity, $H(x,y)=\min(x,y)$, for any $(x,y)\in [0,1]^2$. Similarly, $H(a)=1,H(b)=1$, implies that $H(x,y)=\max(x,y)$; while $H(a)=1,H(b)=0$, implies that $H(x,y)=y$, and $H(a)=0,H(b)=1$, implies that $H(x,y)=x$, again, for $(x,y) \in [0,1]^2$.
\end{proof}

We shall take advantage now of a standard fact concerning Boolean function in $\{0,1\}^n$, to wit, any \textit{monotone} Boolean function $F$ in $\{0,1\}^n$ may be represented as
\begin{equation}\label{equation:monotone_representation}
F=\max\Big(\min\limits_{A_1}, \min\limits_{A_2},\ldots, \min\limits_{A_k},\Big)\, ,
\end{equation}
where $\mathcal{S}=\{A_1, A_2, \ldots, A_k\}$  is a Sperner family in $\{1,\dots,n\}$. To see why this is true, with the usual identification of subsets of $\{1,\dots,n\}$ with elements of $\{0,1\}^n$, the~$A$ comprising the Sperner family  are precisely the minimal subsets $A$ under the action of $F$; minimal  meaning that $F(A)=1$, while for any proper subset $B \varsubsetneq A$ one has $F(B)=0$.

Let $H$ be any selector. By Lemma \ref{lemma:continuous_selector_monotone} we know that $H$ is monotone. The restriction~$F$ of~$H$ to $\{0,1\}^n$ is a monotone Boolean function. Let $\mathcal{S}$ be the Sperner family of the representation \eqref{equation:monotone_representation}, and consider the selector $H_{\mathcal{S}}$. These two selectors $H$ and $H_{\mathcal{S}}$ coincide  on~$\{0,1\}^n$, and so by Lemma \ref{lemma:continuous_selectors_unit_cube} we conclude that $H$ and $H_{\mathcal{S}}$ coincide. We have proved:

\begin{theorem}\label{theor:selectors equal sperner}
Any selector is a Sperner statistic, and conversely.
\end{theorem}

%Let $H$ be a selector. For each $A \subset \{1,\dots,n\}$, identify $A$ with a list in $\{0,1\}^n$, and if $H(A)=1$ declare $A$ minimal if $H(B)=0$ for any proper subset $B$ of $A$. The minimal $A$'s of $H$ conform the associated Sperner family $\mathcal{S}=\{A_1, \ldots, A_k\}$ so that $H\equiv H_{\mathcal{S}}$. Moreover, the  module $h$ of $H$ is given by
%$$
%h(t)=1-\sum_{i}(1-t)^{|A_i|}+\sum_{i<j} (1-t)^{|A_i \cup A_j|}-\dots
%$$

Combining Theorems \ref{theor:continuous conservative equal selector} and \ref{theor:selectors equal sperner}, we have:
\begin{theorem}\label{theor:continuous conservative equal selectors equal sperner}
Let $H\colon \mathbb{R}^n\to\mathbb{R}$. The following are equivalent:
\begin{itemize}\itemsep=0pt
\item[{\rm i)}] $H$ is a continuous conservative statistic;

\item[{\rm ii)}] $H$ is a selector;

\item[{\rm iii)}] $H$ is a Sperner statistic.
\end{itemize}
\end{theorem}

\begin{remark}
{\upshape As a complement to Theorem \ref{theor:continuous conservative equal selectors equal sperner}, it would be interesting
to determine which (noncontinuous) statistics with the selecting property are conservative statistics; and also
to determine which noncontinuous conservative statistics have $h(t)=t$ as module.

Even further, it could be the case that conservative statistics are of just of two types: either selectors, or else statistics satisfying the selecting property  and obtained via a symmetrization procedure akin to the one described in the example following Lemma \ref{lemma:selectorcontinuousimpliesconservative}. Those of the second class have always the identity as module, while among the selectors only the projections have the identity as module.}
\end{remark}

\section{Sperner polynomials}\label{sec:sperner polynomials}
For the analysis of the iteration of Sperner statistics, it is most convenient to consider, instead of its module, the following (dual) polynomial associated to a Sperner statistic.

Let $\mathcal{S}=\{A_1,\dots, A_k\}$ be a Sperner family, and let $H_\mathcal{S}$ and $h_\mathcal{S}$ be its associated statistic and module, respectively. We have already seen that defining
$$
a_k= \#\{B\in \mathbb{B}^n : \text{$B$ has $k$ zeros and $H_\mathcal{S}(B)=0$}\},\quad\text{for $1 \le k\le  n$,}
$$
then
$$
h_\mathcal{S}(t)=\sum_{k=0}^n a_k\, t^k\, (1-t)^{n-k}\,.
$$

Define now
\begin{equation}
\label{def of bk}
b_k= \#\{B\in \mathbb{B}^n : \text{$B$ has $k$ ones and $H_\mathcal{S}(B)=1$}\}.
\end{equation}
Observe that ${n\choose k}-b_k=a_{n-k}$, and $b_0=0$, $b_n=1$. The \textbf{Sperner polynomial} of $H_\mathcal{S}$ is defined as
\begin{equation}\label{def poly g}
g_\mathcal{S}(x):=\sum_{k=0}^n b_k\, x^k\, (1-x)^{n-k}
%=
%\sum_{k=0}^n {n\choose k}\, x^k\, (1-x)^{n-k}- \sum_{k=0}^n a_{n-k}\, x^k\, (1-x)^{n-k}
%\\
%&
%1-h(1-x).
\end{equation}
Observe that
\begin{equation}
\label{relation h and g}g_\mathcal{S}(x)=1-h_\mathcal{S}(1-x),
\end{equation}
so $0<g_\mathcal{S}(x)<1$ for all $x\in(0,1)$, by Lemma \ref{lemma:bounds zero one}.

Notice also that, for each $p\in(0,1)$,
\begin{equation}
\label{g as expectation}
g_\mathcal{S}(p)=\E_p(H_\mathcal{S}),
\end{equation}
where the expectation is taken with respect to the Bernoulli measure $\mu_p$.

\begin{remark}
{\upshape The Sperner polynomial of $H_\mathcal{S}$ is, in fact, the module of the dual selector
$$
\widehat{H}_\mathcal{S}(x_1,\dots, x_n)=1-H_\mathcal{S}(1-x_1,\dots, 1-x_n).
$$
Alternatively, the dual selector can be written as
$$
\widehat{H}_\mathcal{S}=\min\big(\max_{A_1}, \dots, \max_{A_k}\big),
$$
interchanging the roles of minima and maxima.

}
\end{remark}

\begin{remark}{\upshape The list of coefficients $(b_0,b_1,\dots, b_n)$ is sometimes called the profile of (the upset $\mathcal{U}$ associated to) $\mathcal{S}$.

For Sperner polynomials we have that $b_0=0$, $b_n=1$, and $0 \le b_k \le \binom{n}{k}$  for $k=0, 1, \ldots, n$. Besides,
the local LYM inequality (see, for instance, Chapter 3 of \cite{Bo}) yields that the sequence $b_k/{n\choose k}$ is increasing.

It would be interesting to determine which polynomials $g$ are Sperner polynomials. In other terms, to ``characterize'' the profile-polytope of upsets of Sperner families. See \cite{En}.
}
\end{remark}

We show now a useful recurrence relation for Sperner polynomials. It is really a restatement of Lemma~\ref{lemma:recursion} in terms of Sperner polynomials, with equality issues dealt with.
For convenience, only for this lemma, we include the constant functions $g\equiv 1$ and $g\equiv 0$ as (degenerate) Sperner polynomials.

\begin{lemma}
\label{lemma:recurrence for g}
For any Sperner polynomial $g$ of dimension $n$ and for any $t\in(0,1)$, the following recurrence holds:
\begin{equation}\label{eq:recurrence for g}
g(t)=t\, g_1(t)+(1-t)\, g_0(t),
\end{equation}
where $g_1$ and $g_0$ are Sperner polynomials with dimension less than $n$, and $g_1(t)\ge g_0(t)$ for all $t\in[0,1]$.

Moreover, if  $g_1(t^*)=1$ and $g_0(t^*)=0$ for some $t^*\in(0,1)$, then $g$ is the identity. If $g_1(t^*)=g_0(t^*)$ for some $t^*\in(0,1)$, then $g\equiv g_1\equiv g_0$.
\end{lemma}

\begin{proof}
Take independent Bernoulli random variables $I_1,\dots, I_n$ with success probability $t$. Conditioning on the value of $I_1$, for any $t\in(0,1)$,
\begin{align*}
g(t)&=\P(H(I_1,\dots, I_n)=1)
\\
&=t\, \P(H(I_1,\dots, I_n)=1|I_1=1)+(1-t)\,\P(H(I_1,\dots, I_n)=1|I_1=0)
\\
&=t\,g_1(t)+(1-t)\, g_0(t).
\end{align*}

Observe that, if for some $t^*\in(0,1)$, $g_1(t^*)=1$ and $g_0(t^*)=0$, then in fact $g_1(t)=1$ and $g_0(t)=0$ for all $t$, and $g$ is the identity.

Write
$$
H_1(x_2,\dots, x_n)=H(1,x_2,\dots, x_n)\quad\text{and}\quad H_0(x_2,\dots, x_n)=H(0,x_2,\dots, x_n).
$$
Both $H_1$ and $H_0$ are monotone Boolean functions. Observe that
$$
g_1(t)=\E_t(H_1)\quad\text{and}\quad g_0(t)=\E_t(H_0)
$$
(expectations in $n-1$ dimensions). Notice that $H_1\ge H_0$, by the monotonicity of $H$. Therefore, $g_1(t)\ge g_0(t)$ for all $t\in[0,1]$.
Now, if for some $t^*\in(0,1)$, $g_1(t^*)=g_0(t^*)$, then $H_1\equiv H_0$, since $\mu_t$ gives positive mass to all the atoms in $\mathbb{B}^n$. Consequently, $g_1\equiv g_0$.
    \end{proof}

The following example illustrates some alternative ways of calculating modules and Sperner polynomials.

\begin{example}
{\label{tree}\upshape For $n=3$, say that $\mathcal{S}=\{\{1,2\}, \{2,3\}\}$, so that
$$
H(x_1,x_2,x_3)=\max\big(\min(x_1,x_2), \min(x_2,x_3)\big).
$$
Following \eqref{formulamodulesperner}, we could write
$$
h(t)=1-2(1-t)^{2}+(1-t)^{3},
$$
so that
$$
g(t)=2t^2-t^3.
$$

Alternatively, observe that $H=1$ for $(1,1,0)$ and $(0,1,1)$, and also for $(1,1,1)$ (that is, $H=1$ in the upset associated to $\mathcal{S}$):
\begin{center}
\resizebox{10cm}{!}{\includegraphics{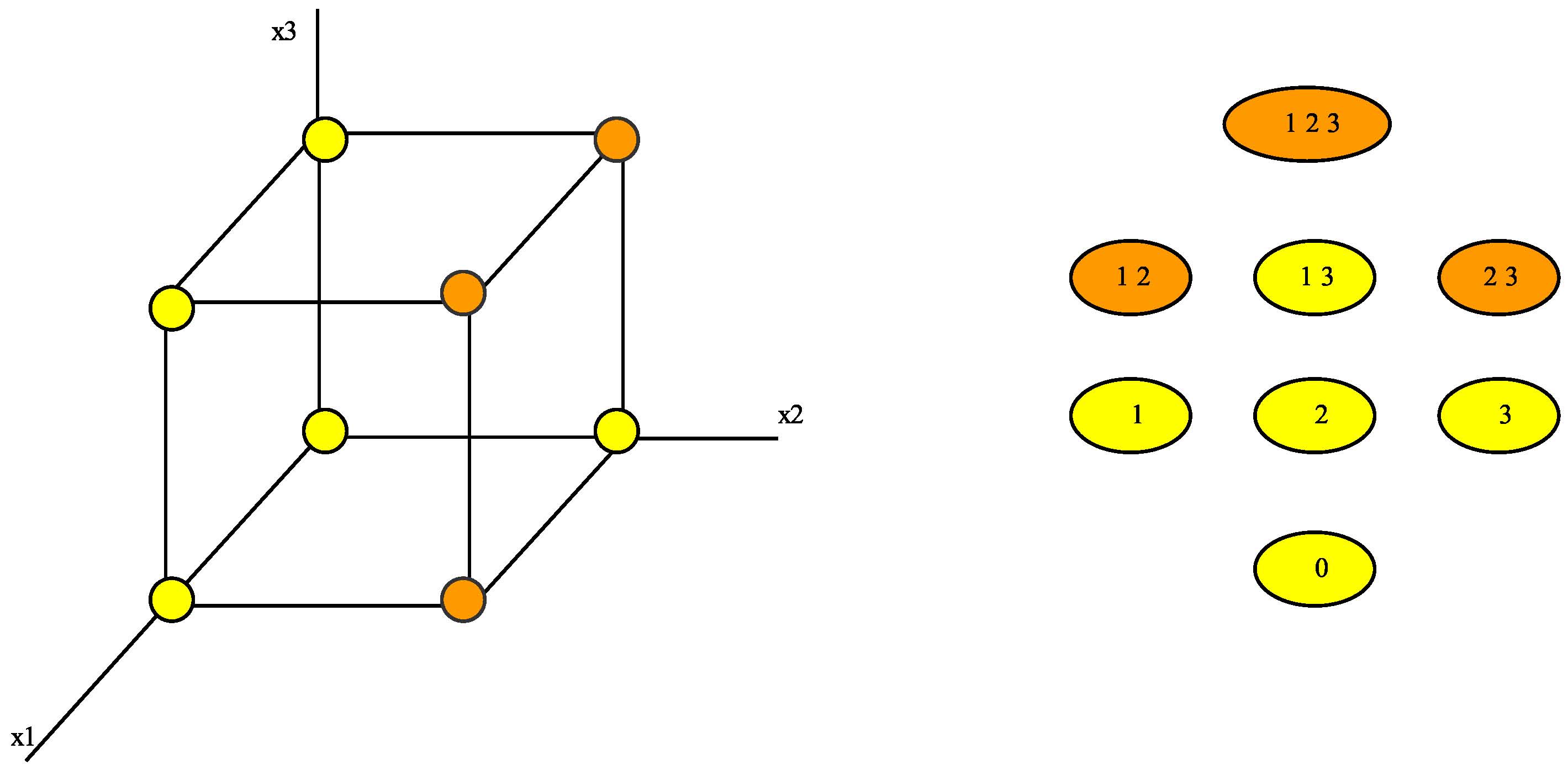}}
\end{center}
Then, using \eqref{def poly g}, as $b_3=1$ and $b_2=2$, we get
$
g(t)=2\, t^2\,(1-t)+t^3=2\, t^2-t^3.
$

With Lemma \ref{lemma:recurrence for g}, we could reinterpret this calculation in a binary tree, as follows: the leaves are labeled 1 (if $H=1$) or 0 otherwise. Then proceed backwards applying  \eqref{eq:recurrence for g}:
\begin{center}
\resizebox{10cm}{!}{\includegraphics{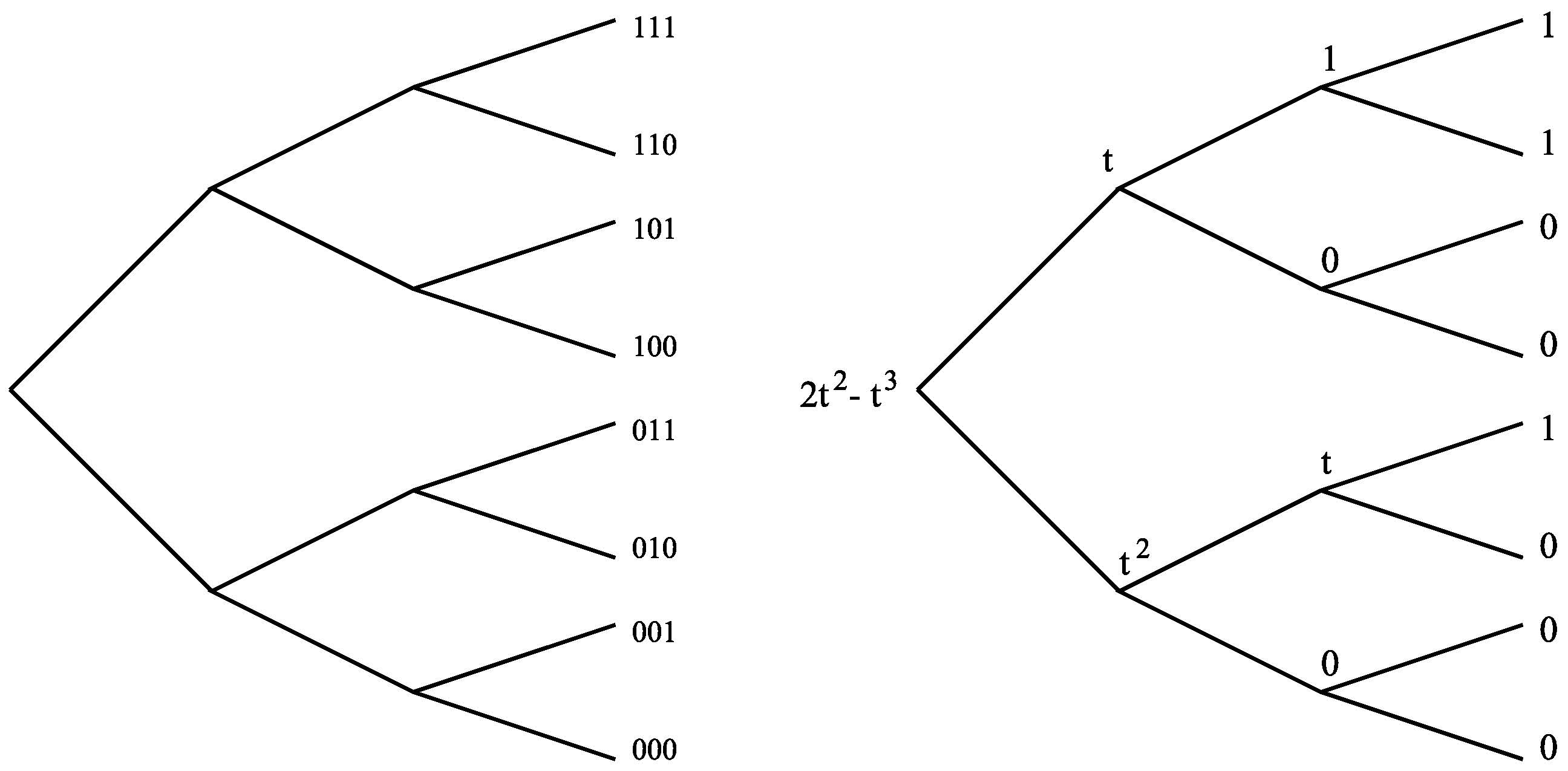}}
\end{center}
}\end{example}

Here is a bound on the derivative of a Sperner polynomial that shall be useful in the sequel.
\begin{lemma}
\label{isoperimetric inequality for g}
For any Sperner polynomial $g(t)$,
\begin{equation}
\label{g' ge than}
g'(t)\ge \frac{g(t)\, (1-g(t))}{t\,(1-t)} \quad\text{for $t\in(0,1)$.}
\end{equation}
If for some $t^*\in(0,1)$ there holds equality in \eqref{g' ge than}, then $g$ is the identity.
\end{lemma}
\begin{proof}
We prove the claim by induction (in the dimension of the Sperner polynomial).

Recall \eqref{eq:recurrence for g}. Notice that, for $t\in(0,1)$,
$$
g(t)=tg_1(t)+(1-t)\, g_0(t)\quad\text{and so}\quad g'(t)=t\, g_1'(t)+(1-t)\, g'_2(t)+(g_1(t)-g_0(t)),
$$
and by the induction hypothesis,
\begin{align*}
&g'(t)\ge \frac{1}{t\,(1-t)}\big[t\, g_1(t)\, (1-g_1(t))+(1-t)\, g_0(t)(1-g_0(t))+t\, (1-t)\, (g_1(t)-g_0(t))
\\
&=\frac{1}{t\,(1-t)}\, \big[t\,g_1(t)+(1-t)\, g_0(t)-\big(t\,g_1^2(t)+(1-t)g_0^2(t)-t(1-t) (g_1(t)-g_0(t))\big)\big]
\\
&\ge \frac{1}{t\, (1-t)}\, \big[(t\,g_1(t)+(1-t)\, g_0(t))\cdot \big(1-(t\,g_1(t)+(1-t)\, g_0(t)\big)\big]=\frac{g(t)\,(1-g(t))}{t\,(1-t)}.
\end{align*}
For the last inequality, observe that
$$
tg_1^2(t)+(1-t)\, g_0^2(t)-t\,(1-t)\, (g_1(t)-g_0(t))\le (t\, g_1(t)+(1-t)\, g_0(t))^2
$$
because $g_1(t)-g_0(t)\ge (g_1(t)-g_0(t))^2$.

We take as the base step for induction the case $n=2$. There are three possible Sperner polynomials: $g(t)=t^2$, $g(t)=t^2+t(1-t)=t$ and $g(t)=t^2+2t(1-t)$. In all of them, the claim is satisfied.

\smallskip
If for some $t^*\in(0,1)$ there holds equality in \eqref{g' ge than}, then
$$
g_1(t^*)-g_0(t^*)=(g_1(t^*)-g_0(t^*))^2;
$$
this implies that $g_1(t^*)-g_0(t^*)$ is equal to 1 or 0.

If $g_1(t^*)-g_0(t^*)=1$, then $g_1(t^*)=1$ and $g_0(t^*)=0$, and we conclude that $g$ is the identity
(see Lemma~\ref{lemma:recurrence for g}).
%In the first case, we have that $g_1(t^*)=1$ and $g_0(t^*)=0$, and we deduce that in fact $g(t)=1$ and $g_0(t)=0$ for all $t$. Thus $g$ is the identity.

If $g_1(t^*)=g_0(t^*)$ then $g(t)=g_1(t)=g_0(t)$ for all $t$, again by Lemma~\ref{lemma:recurrence for g}, and equality holds in \eqref{g' ge than} for $t=t^*$ and $g$ replaced by $g_1$. Iterating the argument, we conclude that $g$ is the identity.
\end{proof}

%\subsection{Convolution of conservative statistics}  We define the (non commutative) \textit{convolution} $H \Box J$ of two conservative statistics $H$ and $J$ of dimensions $n$ and $m$, respectively, as the (conservative) statistic of dimension $n \times m$ given by
%$$
%(H\Box J)(x_1, \ldots, x_{nm})=H\big(J(x_1, \ldots, x_m), J(x_{m+1}, \ldots, x_{2m}), \ldots, J(x_{nm-m+1}, \ldots, x_{nm})\big)\, $$
%If $h$ and $j$ are the respective modules of $H$ and $J$, the module of $H\Box J$ is given by the composition $h \circ j$.
%
%\

\section{Fixed points and iteration of modules of selectors}\label{sec:fixed points of selectors}

Let $H$ be a selector. As we have seen, there is a Sperner family $\mathcal{S}=\{A_1, \ldots, A_k\}$ so that  $H \equiv H_{\mathcal S}$.
Write $h_\mathcal{S}$ for the module. %The repellent fixed points of $h_\mathcal{S}$ will be called \textit{Sperner points}.

We analyze now the fixed points of $h_\mathcal{S}$ in $[0,1]$, which play a crucial role  in the limit theorem of Section~\ref{section:limit theorem}. Trivially, $t=0$ and $t=1$ are always fixed points.

In the following cases, either the module $h_{\mathcal{S}}$ is the identity (all points are fixed), or~$h_\mathcal{S}$ has no fixed points in $(0,1)$ (see Lemmas~\ref{lemma:module identity equal projection}, \ref{derivadas en 0 y 1} and \ref{lemma:basis for recursion}):
\begin{description}\itemsep=0pt
\item[identity] If $\mathcal{S}$ contains a singleton and $\bigcap_{j=1}^k A_j\neq\emptyset$ then $k=1$. %and $A_1$ is a singleton.
In this case $H_{\mathcal{S}}$ is a projection and $h_{\mathcal{S}}$ is the identity. %This case has no Sperner point.
\item[lower] If $\mathcal{S}$ contains a singleton and $\bigcap_{j=1}^k A_j=\emptyset$, then $h_{\mathcal{S}}(x)< x$, for each $x \in (0,1)$. %In this case, the Sperner point is 1.
\item[upper] If $\mathcal{S}$ contains no singleton and $\bigcap_{j=1}^k A_j\neq\emptyset$, then $h_{\mathcal{S}}(x)> x$, for each $x \in (0,1)$. (This includes the case $\mathcal{S}=\{A_1\}$, with $|A_1|\ge 2$).
%The Sperner point is 0.
\end{description}

Apart from these cases, the module of any selector has a \textit{unique} fixed point in $(0,1)$.
\begin{theorem}\label{conjecture:unique fiexd point}
Let $\mathcal{S}=\{A_1, \ldots, A_k\}$ be a Sperner family, with $k\ge 2$.
If each $|A_j| \ge 2$ and~$\cap_{j=1}^k A_j=\emptyset$, the module $h_{\mathcal{S}}$ has a unique fixed point in $(0,1)$, that happens to be repellent.%, is the Sperner point of $h_\mathcal{S}$.
\end{theorem}

We will refer to this fixed point as the \textit{Sperner point} $\omega_H$ of $H$. (In the lower case, $\omega_H=1$; in the upper case, $\omega_H=0$; conventionally, the identity has no Sperner point.)

\begin{proof}[Proof of Theorem {\upshape\ref{conjecture:unique fiexd point}}]
Write simply $h$ and $g$ for $h_\mathcal{S}$ and $g_\mathcal{S}$, respectively.
Recall  (Lemma~\ref{derivadas en 0 y 1}) that $h'(0)=h'(1)=0$. So $h(t)$ must have fixed points in $(0,1)$. If we prove the following:
\begin{equation}
\label{uniqueness condition for h}
\text{if for $t\in(0,1)$, $h(t)=t$, then $h'(t)>1$,}
\end{equation}
then we would get the uniqueness of the fixed point.

%Write
%\begin{align*}
%a_k&=\#\{B\in\mathbb{B}^n: \text{$B$ has $k$ zeros and $H(B)=0$}\},
%\\
%b_k&= \#\{B\in \mathbb{B}^n : \text{$B$ has $k$ ones and $H(B)=1$}\}.
%\end{align*}
%Observe that $a_0=0$ and $a_n=1$. As $b_k={n\choose k}-a_{n-k}$, we find that $b_0=0$ and $b_n=1$.
%
%Notice also that $b_1=0$. Take a list $B=(x_1,\dots, x_n)\in\mathbb{B}^n$ with $x_i=1$ for a unique $i$. Denote also by $B$ the corresponding subset of $\{1,\dots, n\}$. If $H_\mathcal{S}(B)=1$, then
%$$
%\max\big(\min_{A_1}(B),\dots, \min_{A_k}(B)\big)=1.
%$$
%This means that $\min_{A_i}(B)=1$ for some $A_i$, and therefore, $B\supset A_i$. But this is impossible, as the Sperner family does not contain singletons. \ojo{NO S\'{E} SI USAMOS ESTO}
%
%%And vice versa, so that $H(B)=1$ if and only if $B\supset A_i$ for some $A_i$ of the Sperner family.
%
%
%The module can be written as
%$$
%h_\mathcal{S}(x)=\sum_{k=0}^n a_k\, x^k\, (1-x)^{n-k}.
%$$
%Notice that $h_\mathcal{S}=a_2\, t^2(1+O(t))$. Recall that  $h'_\mathcal{S}(0)=h'_\mathcal{S}(1)=0$.
%
%Define now
%$$
%g_\mathcal{S}(x)=\sum_{k=0}^n b_k\, x^k\, (1-x)^{n-k}.
%$$
The associated Sperner polynomial
$
g(x)=1-h(1-x),
$
satisfies $g'(x)=h'(1-x)$, and so $g'(0)=g'(1)=0$.

\smallskip
So condition \eqref{uniqueness condition for h} is equivalent to the corresponding condition for $g$:
\begin{equation}
\label{uniqueness condition for g}
\text{if for $t\in(0,1)$, $g(t)=t$, then $g'(t)>1$.}
\end{equation}

%\medskip
Now, Lemma \ref{isoperimetric inequality for g} yields that if $g(t)=t$, then $g'(t)\ge 1$. The case $g'(t)=1$ corresponds to the identity.\end{proof}

\begin{remark}\label{remark:alternative1 to uniqueness}
{\upshape
An alternative argument to prove the uniqueness of the fixed point in $(0,1)$ for any module as in Theorem \ref{conjecture:unique fiexd point} goes as follows.
%Integrating \eqref{g' ge than}, we get
%$$
%\int_a^b \frac{g'(t)}{g(t)\, (1-g(t))}\, dt \ge \int_a ^b \frac{1}{t\, (1-t)}\, dt\,
%$$
%that is,
%$$
%\ln\Big(\frac{g(b)}{1-g(b)}\, \frac{1-b}{b}\Big)\ge \ln\Big(\frac{g(a)}{1-g(a)}\, \frac{1-a}{a}\Big).
%$$
The estimate \eqref{g' ge than} yields that the function
\begin{equation}\label{function alpha}
\alpha(t)=\frac{g(t)}{1-g(t)}\,\frac{1-t}{t}
\end{equation}
is nondecreasing. Also, writing
$$
\frac{g(t)}{t}=\alpha(t)\,\frac{1-g(t)}{1-t}
$$
one observes that
\begin{equation}
\label{limits of alpha}
\lim_{t\downarrow 0} \alpha(t)=g'(0)\quad\text{and}\quad \lim_{t\uparrow 1} \alpha(t)=\frac{1}{g'(1)}.
\end{equation}
So in our case, $\alpha(t)$ ranges from $0$ to $+\infty$.

In fact, $\alpha(t)$ is strictly increasing. If it were not the case, then $\alpha(t)=\widehat\alpha$ in an interval~$I$, that is, $g(t)(1-t)=\widehat\alpha \, t(1-g(t))$ for all $t\in I$. As $g$ is a polynomial, the same relation would hold for all $t\in[0,1]$, and $\alpha(t)$ would be a constant for $t\in[0,1]$. This contradicts \eqref{limits of alpha}.

If there were two fixed points $t_1<t_2$ in $(0,1)$, then $\alpha(t_1)=1=\alpha(t_2)$. This contradiction with the fact that $\alpha$ is strictly increasing ends the argument.

\smallskip

Observe that $g(t)$ can be written as
$$
g(t)=\frac{\alpha(t)\, t}{1+(\alpha(t)-1)\, t},
$$
with strictly increasing $\alpha(t)$. The $S$-shaped graph of $g$  is transversal to the foliation of the figure, given by the values of $\alpha$, from 0 to $+\infty$  (values of $\alpha$ less than 1 correspond with level curves below the diagonal).

\begin{center}
\resizebox{6cm}{!}{\includegraphics{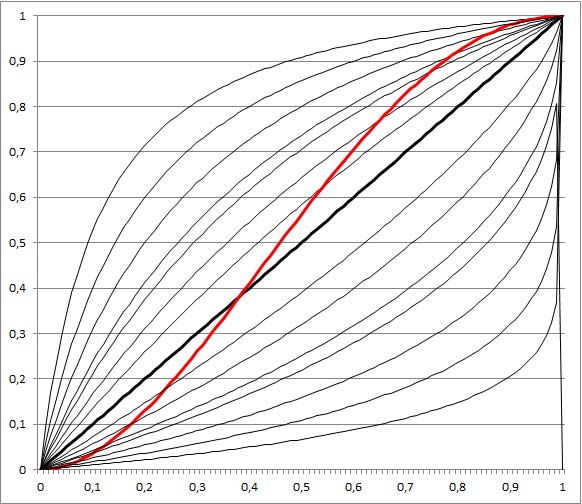}}
\end{center}
% This gives the uniqueness of the fixed point unless $\alpha(t)=1$ for $t$ in an certain interval. But in this case the polynomial $g(t)-t$ would vanish on a segment, and $g$ would be the identity.
}
\end{remark}

\begin{remark}[Fixed points of conservative statistics]
{\upshape The module $h$ of a conservative statistic $H$ (not necessarily continuous) and other than the identity may have in principle more than one fixed point. Those fixed points $p\in[0,1]$ of $h$ with $h'(p)\ge 1$, the \textit{repellent} fixed points,  will play later on, when we consider the iteration of $h$, a role analogous to the Sperner point of selectors
(see Remarks~\ref{remark:iteration of conservative statistics} and~\ref{theorem:limittheorem}). %The collection of repellent fixed points of $h$ is denoted $\mathcal{F}_H$. Observe that, a priori,  0 and 1 may or may not be in $\mathcal{F}_H$.
}\end{remark}

\subsection{An alternative approach to Theorem \ref{conjecture:unique fiexd point}}

An alternative proof of Theorem \ref{conjecture:unique fiexd point} may be written in terms of some well-known results on monotone Boolean functions which can be traced back to \cite{BL}. Let $H$ be a Sperner statistic. Its restriction to the Boolean cube is a monotone Boolean function.
Recall from \eqref{g as expectation} that, for each $p\in(0,1)$,
$$
g(p)=\E_p(H).
$$
Russo's lemma (see \cite{Ru}) asserts that
$$
g'(p)=I_p(H),
%\sum_{i=1}^k I_p(i),
$$
where $I_p(H)$ (the \textit{total influence} of $H$) is defined as
$$
I_p(H)=\sum_{B\in\mathbb{B}^n} \mu_p(B) \, n(B),
$$
and $n(B)$ is the number of neighbours $B'$ of $B$ (differing from $B$ in one coordinate) such that $H(B)\ne H(B')$.

Further, the quantity $I_p({H})$ can be bounded from below as follows
$$
I_p(H)\ge \frac{g(p)}{p}\, \log_p(g(p))
$$
(an \textit{edge isoperimetric inequality}, see formula (3) in \cite{KK}). This yields
\begin{equation}
\label{bound for g'}
g'(p)\ge \frac{g(p)}{p}\, \log_p(g(p))\,.
\end{equation}

If $g(p)=p$, we get that $g'(p)\ge 1$. The case $g(p)=p$ and $g'(p)=1$ corresponds to the identity, as it is shown with the following argument (similar to that in Remark \ref{remark:alternative1 to uniqueness}). Write~\eqref{bound for g'} as
\begin{equation}\label{isoperimetric for qprime}
\frac{g'(p)}{g(p)}\ge \frac{\ln(g(p))}{p\, \ln(p)}, \quad\text{that is}\quad \frac{g'(p)}{g(p)}\, \ln(p)\le \frac{\ln(g(p))}{p}.
\end{equation}
This is equivalent to saying that the function
$$
\alpha(t)=\frac{\ln(g(t))}{\ln(t)}
$$
is nonincreasing because, thanks to \eqref{isoperimetric for qprime},
$$
\alpha'(t)=\frac{1}{\ln(t)^2} \Big[ \frac{g'(t)}{g(t)}\, \ln(t)-\frac{\ln(g(t))}{t}\Big]\le 0.
%=\frac{1}{\ln( t)}\, \Big[ \frac{q'(t)}{q(t)}-\frac{\ln(q(t))}{t\, \ln(t)}\Big]
$$
Notice that, as $\mathcal{S}$ does not contain singletons, $b_1=0$, so $g(t)=b_k\, t^k\,(1+O(t))$ for certain $k\ge 2$ and $b_k\ge 1$. This means that
$$
\lim_{t\downarrow    0} \alpha(t)=k\ge 2.
$$
Observe also that $\lim_{t\uparrow 1} \alpha(t)=0.$ As
$$
g(t)=t^{\alpha(t)},
$$
the values of $\alpha$ (ranging now from $k\ge 2$ to $0$) give rise to a foliation (similar to the one depicted in Remark \ref{remark:alternative1 to uniqueness}) and we end the proof as there.

\begin{remark}\label{remark:alternative3 to uniqueness}
{\upshape The case $p=1/2$ of the above observation (that is, $g(1/2)=1/2$ and $g'(1/2)=1$ together imply that $g$ is the identity) can be dealt with the classical Kruskal--Katona theorem (see \cite{Kr}, \cite{Ka}). Consider an upset $\mathcal{U}$ and denote by $b_k$ the number of sets in $\mathcal{U}$ of size $k$. Write
% is the set
%$$
%\mathcal{A}=\{A\subset \{1,\dots,n\}: A\supset \text{ some $S_j\in\mathcal{S}$}\}.
%$$
%Notice that $h_\mathcal{S}=h_\mathcal{A}$. We will prove the conjecture for $h_\mathcal{A}$.
%In the following, we will fix the upset $\mathcal{A}$ associated to the Sperner family $\mathcal{S}$, and consider the statistic $H_\mathcal{A}$ and the modulus~$h_\mathcal{A}$. The subindexes will be omitted.
$$
|\mathcal{U}|=\sum_{k=0}^n b_k\quad\text{and}\quad ||\mathcal{U}||=\sum_{k=0}^n k\, b_k.
$$
As a consequence of the Kruskal-Katona theorem,
$$
\|\mathcal{U}\|\ge n\,|\mathcal{U}|-\|I(|\mathcal{U}|)\|,
$$
where $I(j)$ denotes the set (the \textit{initial segment} of length $j$) comprising the first $j$ sets in the colex order. Notice that $\|I(2^r)\|=r\, 2^{r-1}$ (the mean size of the subsets of $\{1,\dots, r\}$).

In particular, if $|\mathcal{U}|=2^{n-1}$, then
\begin{equation}
\label{cota para |U|bis}
\|\mathcal{U}\| \ge n\, 2^{n-1} -\frac{n-1}{2}\, 2^{n-1}=2^{n-1}\, \frac{n+1}{2}=\frac{n+1}{2}\, |\mathcal{U}|.
\end{equation}

\smallskip
Consider now the upset $\mathcal{A}=\{B\in\mathbb{B}^n: H(B)=1\}$. If $g(1/2)=1/2$, then
$$
\sum_{k=0}^n b_k\,\frac{1}{2^n}=\frac{1}{2}\quad\Longrightarrow\quad |\mathcal{A}|=\sum_{k=0}^n b_k=2^{n-1}.
$$

Observe that
$$
g'(x)=\sum_{k=0}^n b_k\, \big[k\, x^{k-1}\, (1-x)^{n-k}-(n-k)\, x^k\, (1-x)^{n-k-1}\big]\,,
$$
so
$$
g'(1/2)=\frac{1}{2^{n-1}} \sum_{k=0}^n b_k\, (2k-n)=\frac{1}{2^{n-1}}\, 2\|\mathcal{A}\|-n\ge 1,
$$
thanks to \eqref{cota para |U|bis}.
Equality holds only if $\mathcal{A}$ is the complement of the initial segment of length~$2^{n-1}$. This case corresponds to a projection, in which case we already know that~$h$ is the identity.
}
\end{remark}

\subsection{Location of the fixed points of modules of Zermelo and of order statistics}

Order statistics, which somehow have maximal overlapping, and Zermelo statistics, which have no overlapping at~all, are extreme cases of Sperner statistics. We now analyze the location of their Sperner points.

\subsubsection{Order statistics}

The module of the order statistic $H_{r:n}$, $1 \le r \le n$, is the
polynomial
$$
h_{r:n}(x)=\sum_{j=r}^n B^{(n)}_j(x)\,.
$$
For $r=1$ and $r=n$ the only fixed points of $h_{r:n}(x)$ are 0 and 1. For $1<r<n$, the
the uniqueness of the fixed point of $h_{r:n}$ in $(0,1)$ can be proved by a simple Calculus argument.
\begin{lemma}\label{fixed point of order statistic}
The module of any order statistic $h_{r:n}$, with $1 < r <n$, has a
unique fixed point $\omega_{r:n}$ in $(0,1)$, which besides is repellent.
\end{lemma}

\begin{proof}Fix $1 < r < n$; we shall repeatedly use below that $r\ne 1$ and $r\ne n$. We simplify and denote $h_{r:n}$ by $h$.
Observe that, %$$
%h=\sum_{j=r}^n B^{(n)}_j\, .
%$$
%And, therefore,
using \eqref{equation:derivativeBernstein},
$$
h^{\prime}(x)=n \,B^{(n-1)}_{r-1}(x) =n\, {{n-1}\choose{r-1}}\, x^{r-1}\, (1-x)^{n-r}\, .
$$
So that, $h^{\prime}(0)=h^{\prime}(1)=0$, while $h^{\prime}(x)>0$, for $x\in (0,1)$. Besides,
$$
\frac{h''(x)}{h'(x)}=\frac{(r-1)-x(n-1)}{x(1-x)}\, .
$$
This gives that $h^{\prime}$ is a unimodal density, which increases for $0 \le x \le {(r-1)}/{(n-1)}$, and decreases on ${(r-1)}/{(n-1)} <x < 1$.
Further, observe that
\begin{equation}\label{hprime bigger than 1}
h^{\prime}\Big(\frac{r-1}{n-1}\Big)=n \binom{n-1}{r-1} \frac{(r-1)^{r-1} (n-r)^{n-r}}{(n-1)^{n-1}} >1\, .
\end{equation}
To see this, one may use Stirling's approximation
$$
\sqrt{2\pi} \,n^{n+1/2}\, e^{-n}\le n!\le e\, n^{n+1/2}\, e^{-n} \quad\text{for all $n\ge 1$.}
$$
%\ojo{QUIZ\'{A}S HAYA OTRO ARGUMENTO M\'{A}S DIRECTO. EL M\'{I}NIMO EST\'{A} EN $r\approx n/2$, Y CRECE CON $n$}

From \eqref{hprime bigger than 1} we deduce that there are only two points in $(0,1)$ where $h'$ equals 1, and so the polynomial $h$ has a unique fixed point in $(0,1)$. Besides, since $h^{\prime}(0)=h^{\prime}(1)=0$, this fixed point is repellent.
\end{proof}

\begin{lemma}
\label{location of fixed point of order statistic}
{\rm a)} For any given $n$, the Sperner points $\omega_{r:n}$ increase with $r$:
\begin{equation}\label{equation:omegas-increasing}
0<\omega_{2:n} < \omega_{3:n}< \cdots <\omega_{n-1:n} <1\,,
\end{equation}
and satisfy the symmetry relation
\begin{equation}\label{equation:symmetryrelationomegas}
\omega_{r:n}+\omega_{n-r+1:n}=1\, ,\quad\text{for $1<r<n$.}
\end{equation}

{\rm b)} For any $1 < r < n$,
\begin{equation}\label{equation:errorfixpoints}
\Big|\omega_{r:n} -\frac{2r-1}{2n}\Big| \le
\sqrt{\frac{\ln(n)}{n}}\, .
\end{equation}
\end{lemma}

We refer to \cite{KAlam} for some general results concerning unimodality of order statistics.

Equation \eqref{equation:errorfixpoints} implies that if $n \to
\infty$ and $r/n \to \lambda \in (0,1)$ then $\omega_{r:n} \to
\lambda$.

\begin{proof}
a) Since $h_{r:n}(x)> h_{s:n}(x)$, for any $x\in(0,1)$ and for $1 < r <
s <n$, one deduces that $\omega_{r:n}<\omega_{s:n}$ for $1 < r < s
<n$.

As a consequence of \eqref{Bernstein partition}, these modules satisfy the symmetry relation
\begin{equation}
\label{equation:symmetryrelationmodules}h_{r:n}(x)+h_{n-r+1:n}(1-x)\equiv
1\,.
\end{equation}
This yields \eqref{equation:symmetryrelationomegas}.

\smallskip

b) To estimate  the location of the fixed points, we first observe from \eqref{module order statistic}
that, for any~$x$,
$$
h(x)=\P(\textrm{Bin}(n,x)\ge r)=\P(\textrm{Bin}(n,x)\ge
r-{1}/{2})\,,
$$
and so,
$$
\P\Big(\frac{\textrm{Bin}(n,\omega_{r:n})}{n}-\omega_{r:n}\ge
\underset{:=t_{r:n}}{\underbrace{\frac{2r-{1}}{2n}-\omega_{r:n}}}\Big)=\omega_{r:n}\,
.
$$
Asume that $t_{r:n}\ge 0$. Then from Hoeffding's inequality (see, for instance, %Theorem 2.1 in~\cite{DL} or
Theorem A.1.4 in~\cite{AS}), we
deduce that
$$\omega_{r:n} \le e^{-2nt_{r:n}^2}$$
and, consequently, that
$$
t_{r:n}^2 \le \frac{1}{2n}
\ln{\big(\frac{1}{\omega_{r:n}}\big)}\,.
$$

Assume now that $t_{r:n}\le 0$. Write
$\omega^*_{r:n}=\omega_{n-r+1:n}=1-\omega_{r:n}$. We have
$$
\P\Big(\frac{\textrm{Bin}(n,\omega^*_{r:n})}{n}-\omega^*_{r:n}\ge
\frac{2(n-r+1)-{1}}{2n}-\omega^*_{r:n}\Big)=\omega^*_{r:n}\, ,
$$
and, therefore, that
$$
\P\Big(\frac{\textrm{Bin}(n,\omega^*_{r:n})}{n}-\omega^*_{r:n}\ge
-t_{r:n}\Big)=\omega^*_{r:n}\, .
$$
And now, from Hoeffding's inequality, as above,
$$t_{r:n}^2 \le \frac{1}{2n} \ln{\big(\frac{1}{\omega^*_{r:n}}\big)}$$

We conclude that
$$t_{r:n}^2 \le \frac{1}{2n}
\max\Big(\ln\big(\frac{1}{\omega_{r:n}}\big)\, ,\,
\ln\big(\frac{1}{\omega^*_{r:n}}\big)\Big)\, .$$

We shall show now that $\omega_{2:n}\ge {1}/{n^2}$. This will finish the proof, thanks to \eqref{equation:omegas-increasing}.

%Since an
%analogous argument shows that $\omega_{n-1:n}\le 1 -{1}/{n^2}$, this will finish the proof.

Now, fix $n$ and write $\omega=\omega_{2:n}$; it satisfies
$$
1-\omega=\P(\textrm{Bin}(n,\omega)\le 1)\, ,
$$
or,
$$
1-\omega=(1-\omega)^{n}+n\, \omega\,(1-\omega)^{n-1}\, ,
$$
or,
$$
1=(1-\omega)^{n-1}+n(1-\omega)^{n-2}\omega=(1-\omega)^{n-2} (1+(n-1)\omega)\, .
$$
Let $f$ be defined by $f(x)=(1-x)^{n-2}(1+(n-1)x)$. Now $f(0)=1$, $f(1)=0$, $f$ is positive in $[0,1)$, increases up to $x={1}/{(n-1)^2}$ and decreases thereafter. So, the first point  $x$ after~0 where it reaches the value $1$, which is $x=\omega$, must satisfy $x \ge {1}/{(n-1)^2} > {1}/{n^2}$, as claimed.
\end{proof}

\subsubsection{Zermelo statistics} Again, there is a direct proof of the fact that the modules of Zermelo statistics (except for some exceptional trivial cases), have a unique fixed point in~$(0,1)$.

\begin{lemma}\label{lemma:fixpoint pairwise disjoint sperner module}
Let $k\ge 2$ and let $\alpha_1, \ldots, \alpha_k$ be integer numbers, $\alpha_j \ge 2$, for $j=1, \ldots, k$. The equation
\begin{equation}\label{eq:fixpoint pairwise disjoint sperner module}
\prod_{j=1}^k (1-s^{\alpha_j})=1-s
\end{equation}
has a {\upshape unique} solution for $s \in (0,1)$.
\end{lemma}
\begin{proof}
Consider the function $f$ defined for $t \in [0,1]$ by
$$
f(t)=\frac{1}{1-t} \,\prod_{j=1}^k (1-t^{\alpha_j})\,.
$$
Observe that $f(0)=1$ and $f(1)=0$. We want to show that $f(t)=1$ occurs only at a single $t \in(0,1)$.

For $t \in[0,1]$ we have
$$
f^{\prime}(t)=f(t)\Big(\frac{1}{1-t} -\sum_{j=1}^{k} \frac{\alpha_j \,t^{\alpha_j-1}}{1-t^{\alpha_j}}\Big)\,.
$$
In particular, $f^{\prime}(0)=1$, and therefore $f(t)=1$ for some $t \in (0,1)$.

To show that there is only one solution of \eqref{eq:fixpoint pairwise disjoint sperner module}, it is enough to show that $f^{\prime}$ vanishes at a single point in $(0,1)$, or, equivalently, that the function $g$ given by
$$
g(t)=\sum_{j=1}^{k} \frac{\alpha_j \,t^{\alpha_j-1}(1-t)}{1-t^{\alpha_j}}
$$
takes the value 1 for a unique  $t \in (0,1)$. But observe that  $g$ is increasing, $g(0)=0$ and $g(1)=k >1$.
\end{proof}

In fact, Lemma \ref{lemma:fixpoint pairwise disjoint sperner module} holds for real $\alpha_j>1$.

\medskip

Let $\mathcal{S}=\{A_1, \ldots, A_k\}$ be a  disjoint family with $k \ge 2$. Denote $a_j=|A_j|$, for $1 \le j \le k$. Recall that the module $h_{\mathcal S}$ of its associated Zermelo statistic $H_{\mathcal S}$ is given by
$$
h_{\mathcal{S}}(t)=\prod_{j=1}^k \big(1-(1-t)^{a_j}\big)\,.
$$

If one of the $a_j$ is 1, then $h_{\mathcal{S}}(t) <t$, for each $t \in (0,1)$ and $h_{\mathcal{S}}$ has no fixed point in~$(0,1)$. If each $a_j \ge 2$,  then $t$  is a fixed point of $h_{\mathcal S}$ if and only if $s=1-t$ satisfies~\eqref{eq:fixpoint pairwise disjoint sperner module}. Therefore,
\begin{corollary} For a Zermelo statistic and with the notations above, 
\begin{enumerate}
\item[\rm a)] If $a_j =1$ for some $1 \le j \le k$, then $h_{\mathcal S}$ has no fixed points in $(0,1)$.
\item[\rm b)] If $a_j\ge 2$ for each $1 \le j \le k$, then $h_{\mathcal S}$ has a unique Sperner point in $(0,1)$.
\end{enumerate}
\end{corollary}
%In the second case, the fixed point is necessarily repellent since $h_{\mathcal S}^{\prime}(0)=h_{\mathcal S}^{\prime}(1)=0$.

As for the location of the fixed point, consider, for $k\ge 2$ and $m \ge 2$,  the Zermelo statistics~$Z_{k,m}$ where the  disjoint  family have $k$ members each of size $m$ (so that $n=km$), and denote the unique fixed point of its module $h_{k,m}(t)=(1-(1-t)^{m})^{k}$ by $\eta_{k,m}$.

It can be proved that:
\begin{lemma}
Fix integers $k \ge 2$ and $m \ge 2$. Then, for large $m$ we have that
$$
1-\eta_{k,m} \asymp \frac{1}{k^{1/(m-1)}}.
$$
More precisely,
$$
\lim_{m \to \infty}\sup_{k\ge 2}\Big|\ln\big(\frac{1}{1-\eta_{k,m}}\big)-\frac{\ln(k)}{m-1} \Big|= 0\, .
$$
In fact,
$$
\frac{1}{b(m)} \frac{1}{k^{1/(m-1)}}\le 1-\eta_{k,m} \le b(m) \,\frac{1}{k^{1/(m-1)}}\, ,
$$
where
$$
b(m)= 3 (\ln(m))^{1/{m-1}}\, .
$$
\end{lemma}

\subsection{Iteration  of modules of selectors}\label{sec:iteration of modules}

Let $H$ be a selector (or Sperner statistic), and let  $\mathcal{S}=\{A_1, \ldots, A_k\}$ be the associated Sperner family, so that  $H =H_{\mathcal S}$.
Write $h=h_\mathcal{S}$ for the module, and $h^{(N)}$ for the composition of $h$ with itself $N$ times.

To analyze the asymptotic behaviour of $h^{(N)}$ as $N\to\infty$, we distinguish, as in the beginning of Section \ref{sec:fixed points of selectors},  four possibilities:
\begin{description}
\item[identity] \textit{$h$ is the identity}. In this case, $h^{(N)}$ is the identity for any $N \ge 1$. Recall that this occurs only when the Sperner family consists of a singleton (so $H$ is a projection).
    %In this case, both $h^{\prime}(1)>0$, $h^{\prime}(0)>0$, actually both derivatives are 1.

    \item[lower] \textit{$h(t)<t$, for any $t \in (0,1)$}. Recall that this occurs precisely when $\mathcal{S}$ contains a singleton and $k \ge 2$. Observe that $h^{\prime}(0)=0$ and $h^{\prime}(1)>0$. In this case,
                $$
       h^{(\infty)}(t):=\lim_{N \to \infty} h^{(N)}(t)=\begin{cases}
       0, & \text{for} \  0\le t <1\, ,\\
       1, & \text{for} \  t =1\, .
              \end{cases}
        $$
%
%
 %       $$
 %       \lim_{N \to \infty} h^{(N)}(t)=0, \ \ \text{for $0 \le t <1$}, \quad\text{while}\quad \lim_{N \to \infty} h^{(N)}(1)=1.
 %       $$
        %In this case, $L_H$ is the constant $1$.

\item[upper] \textit{$h(t)>t$, for any $t \in (0,1)$}.  Here,  $h^{\prime}(0)>0$ and $h^{\prime}(1)=0$. This occurs when $\bigcap_{j=1}^k A_j$ is non empty, and $\mathcal{S}$ contains no singleton. In this case,
                   $$
       h^{(\infty)}(t)=\lim_{N \to \infty} h^{(N)}(t)=\begin{cases}
              0, & \text{for} \  t =0\, ,
              \\
1, & \text{for} \  0<  t \le 1\, .\\
              \end{cases}
        $$

%     $$
%\lim_{N \to \infty} h^{(N)}(t)=1,\ \ \text{for $0 < t \le 1$},\quad\text{while}\quad \lim_{N \to \infty} h^{(N)}(0)=0.
%$$
%In this case, $L_H$ is the constant $0$.

    \item[fixed point in $(0,1)$] Here $h^{\prime}(0)=0$ and $h^{\prime}(1)=0$, and $h$  has a unique fixed point $\omega$ in $(0,1)$, which is repellent. We have
        $$
       h^{(\infty)}(t)=\lim_{N \to \infty} h^{(N)}(t)=\begin{cases}
       0, & \text{for} \  0\le t <\omega\, ,\\
       \omega , & \text{for} \  t=\omega  \, ,\\
       1, & \text{for} \  \omega < t \le 1\, .
              \end{cases}
        $$
%        Now $L_H$ is the constant $\omega$.
\end{description}
\begin{figure}[h]
\centering\resizebox{6.3cm}{!}{\includegraphics{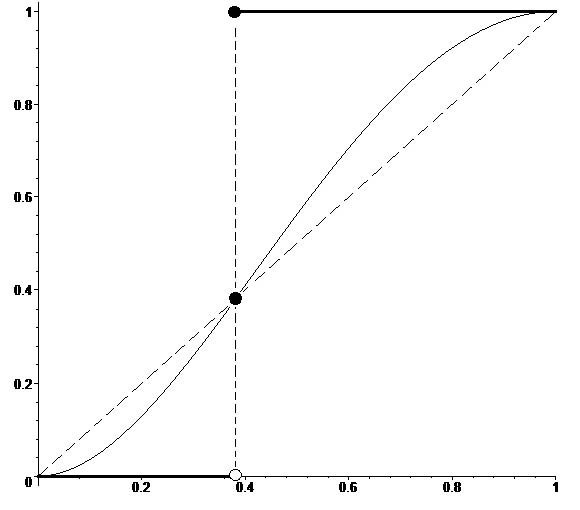}}
\caption{The limit function $h^{(\infty)}$ for $h(t)=(1-(1-t)^2)^2$.}\label{fig:hinfty}
\end{figure}

\begin{remark}[Iteration of conservative statistics]\label{remark:iteration of conservative statistics}
{\upshape Let $h$ be the module of a conservative statistic $H$. Recall (Lemma \ref{lemma:conservative_implies_selector}) that $H$ satisfies the selecting property; no continuity is assumed (or required) here. %We are interested in the asymptotic behavior of the $N$-fold composition $h^{(N)}$ of~$h$ with itself, which we shall %use to determine the asymptotic behavior (in distribution) of $\HN(X)$  for any $X$.

%\smallskip
If $h(x)=x$, for every $x \in [0,1]$, then likewise $h^{(N)}(x)=x$, for every  $x \in [0,1]$.

%\smallskip

If $h$ is not the identity, then $h$ has a finite number of fixed points.
Let $a_0=0 < a_1 <a_2 < \cdots < a_k < a_{k+1}=1$, be the fixed points of $h$.
We term an (open) interval $(a_j, a_{j+1})$ determined by consecutive fixed points an \textit{up} interval if $h(x)>x$ for every $x \in (a_j, a_{j+1})$; otherwise, if $h(x)<x$ for every $x \in (a_j, a_{j+1})$, we call it a \textit{down} interval.
In an up interval  $(a_j, a_{j+1})$, we have that $\lim_{N\to \infty}h^{(N)}(x)=a_{j+1}$, for every $x \in  (a_j, a_{j+1})$, while in  a down interval $(a_j, a_{j+1})$, we have that $\lim_{N\to \infty}h^{(N)}(x)=a_{j}$.

Consequently, except for a finite number of points $x$, we have that $h^{(N)}(x)$ converges, as $N \to \infty$,  to $h^{(\infty)}(x)$, where $h^{(\infty)}$ is the distribution function of a random variable $L_H$ which takes as values only some of the fixed points of $h$, precisely those where $h^{\prime}(x) \ge 1$.  This random variable $L_H$ depends only on $H$.
}
\end{remark}

\section{A limit theorem for selectors}\label{section:limit theorem}

Let $H$ be a selector (or Sperner statistic) of dimension $n$ and module $h$.
We are interested in the asymptotic behavior of the repeated application of $H$ to random samples.

Recall that, for a random variable $X$, $\H(X)=H(X_1,\dots, X_n)$, where the $X_j$ are independent copies of $X$. Now define
$$
\H^{(N)}(X)=\H(\H^{(N-1)}(X))\quad\text{for $N\ge 2$}
$$
(of course, $\H^{(1)}=\H$). Observe that $\H^{(N)}$ acts on $n^N$ independent copies of $X$.

%By repeated application of $H$ we mean the following: we define $\H^{(2)}(X)=\H(\H(X))$, that is,
%\begin{align*}
%\H^{(2)}(X)%&=\H(\H(X))=\H(H(X_1,\dots, X_n))
%%\\
%%&
%=H(H(X_1,\dots, X_n), H(X_{n+1},\dots, X_{2n}),\dots, H(X_{n^2-n+1},\dots, X_{n^2})),
%\end{align*}
%where the $X_j$'s are independent copies of $X$. The subsequent iterates are defined analogously.

By its very definition, for any random variable $X$ the distribution function $F_{\HN(X)}$ of $\HN(X)$ is given by $h^{(N)}\circ F_X$, where $h^{(N)}$ denotes the composition of $h$ with itself $N$ times.

%We move now to Sperner statistics $H_\mathcal{S}$, with $\mathcal{S}=\{A_1,\dots, A_k\}$.

For projections, we already know that $h$ is the identity, and we get:
$$
\mathbf{H}^{(N)}(X)\stackrel{\rm d}{=}X, \quad\text{for any $N \ge 1$ and any random variable~$X$.}
$$

%we have:
%\begin{description}\itemsep=0pt
%\item[identity:] If $h$ is the identity, then
%\item[lower:] If $h(t) <t$ for any $t \in (0,1)$, then $\mathbf{H}^{(N)}(X)\stackrel{\rm d}{\longrightarrow} Q_X(1)$, for any random variable $X$.
%\item[upper:] If $h(t) >t$ for any $t \in (0,1)$, then $\mathbf{H}^{(N)}(X)\stackrel{\rm d}{\longrightarrow} Q_X(0)$, for any random variable $X$.
%\item[fixed point:] In this remaining case, $\mathcal{F}_H$ (the repellent fixed points of $h$) is a finite set in $(0,1)$ (both $0$ and $1$ are not repellent). When there is only one such fixed point $\omega$ (as we believe it is always the case), then
%    $$\mathbf{H}^{(N)}(X)\stackrel{\rm d}{\longrightarrow} Q_X(\omega)\, , \quad  \mbox{for any random variable $X$.}
%    $$
%    In other terms,   for any given $X$ and for large $N$,  the variable $H^{(N)}$ is approximately the $\omega$-quantile of $X$.
%\end{description}

\begin{theorem}\label{main th Sperner}
Let $H$ be a Sperner statistic different from a projection. Let $\omega_H$ be its Sperner point. Then for any random variable $X$, we have
$$
\mathbf{H}_N(X) \stackrel{d}{\longrightarrow} Q_X(\omega_H)\, .
$$
\end{theorem}

This follows readily from the discussion in Section~\ref{sec:iteration of modules}.

Recall from Section \ref{sec:fixed points of selectors} that the Sperner point could be 1 (when $h(x)<x$, that is, if $\mathcal{S}$ contains a singleton and $\cap_{j=1}^k A_j=\emptyset$); or 0 (when $h(x)>$, that is, if $\mathcal{S}$ contains no singleton and $\cap_{j=1}^k A_j\neq\emptyset$). In the remaining cases, the Sperner point belongs to $(0,1)$.

%By $Z$ non trivial, we mean that the Sperner family contains no singleton, and more that one subset of $\N_n$. In those trivial cases, one also has convergence, to $Q_X(1)$, in the first case, and to $Q_X(0)$ in the second.

The Zermelo max-min statistic $M\!m$ of equation \eqref{equation:maxminstatistic} applied repeatedly to any variable~$X$ converges to $Q_X(1-1/\varphi)$.
 If $X$ takes the values $a<b$ with respective probabilities $p\in(0,1)$ and $1-p$, then
 $$
 \mathbf{M\!m}^{(N)}(X) \stackrel{\rm d}{\longrightarrow} \begin{cases}
 a, & \quad \mbox{{if $1-1/\varphi<p$}}\,,\\
 X, & \quad \mbox{{if $1-1/\varphi=p$}}\, ,\\
 b, &\quad \mbox{{if $1-1/\varphi>p$}}\,.
 \end{cases}
 $$
 This is just the example discussed in Section \ref{randomizingzermelo} of this paper.

\smallskip
 For $X$ a standard normal random variable, $\mathbf{M\!m}^{(N)}(X)$ converges in distribution to the constant $\Phi^{-1}(1-1/\varphi)=\Phi^{-1}((3-\sqrt{5})/2)$, where $\Phi$ denotes the distribution function of a standard normal random variable.

\begin{remark}[Fixed points of $\H$]\label{remark:fixed points of H}
{\upshape Let $H$ be a selector. We say that a random variable $X$ is a fixed point of the operator $\H$ if $\H(X)\stackrel{\rm d}{=} X$. Observe that the constants are (trivial) fixed points of $\H$. Theorem \ref{main th Sperner} says that the only (non trivial) fixed points of $\H$ are the random variables $X$ taking two values $a<b$, with respective probabilities $\omega_H$ and $1-\omega_H$.}
\end{remark}

\begin{remark}[Limit theorem for conservative statistics]
\label{theorem:limittheorem} {\upshape Let $H$ be a conservative statistic whose module $h$ is not the identity, and let~$X$ be any random variable. An analogue of Theorem \ref{main th Sperner} in this case would read:
the sequence $\{\mathbf{H}^{(N)}(X)\}$  converges in distribution to a finite random variable which is  a mixture of the quantiles
$\{Q_X(\omega): \omega \in\mathcal{F}_H)\}$ of the repellent fixed points of $h$.}
\end{remark}

\subsection{Rate of convergence}

Let $h$ be the module of a selector $H$. Let $U$ be a uniform variable. Recall that  $h$ is the distribution function of $\mathbf{H}(U)$; and, in general,  $h^{(N)}$ is the distribution function of $\mathbf{H}^{(N)}(U)$.

%If $h$ is not the identity, then $h$ has a finite number of repelling fixed points.
%%The $N$-fold composition $h^{(N)}$ of $h$ with itself is the module of the $N$-fold convolution $H^{(N)}$ of $H$ with itself.
%We have seen that $h^{(N)}$ converges to $h^{(\infty)}$ (at the continuity points); this $h^{\infty}$ is the (restriction to $[0,1]$ of the) distribution function of a random variable which takes just a finite number of values, all in $[0,1]$.

Suppose that $h$ is not the identity. The sequence $\{h^{(N)}\}_{N}$ does not converge to $h^{\infty}$ in the sup-norm (Kolmogorov metric). The next lemma specifies the rate of convergence of  $\{h^{(N)}\}_{N}$ to $h^{\infty}$ in the $L^1$ norm (Wasserstein metric).

\begin{lemma}\label{lemma:rate_convergence} For any module $h$ as above,
\begin{equation}
\int_{0}^{1} \big|h^{(N)}(x)-h^{(\infty)}(x)\big| dx= O\Big(\frac{1}{N^{\eta}}\Big)
\end{equation}
for some $\eta >0$.
\end{lemma}
The distribution function of $\H^{(N)}$  applied to $n^N$ uniform independent variables is $0$ for $x <0$ and $1$ for $x>1$, and the same is true for the distribution function of the limit random variable. That is why the integral above (just on $[0,1]$) gives the Wasserstein distance.

If $h$ is ``lower'', Lemma \ref{lemma:rate_convergence} follows directly from the following lemma, while the ``upper'' case is analogous; for the case with one fixed point in $(0,1)$, it is enough to
split $[0,1]$ into two intervals and rescale these arguments.

%\begin{lemma}
%Let $g$ be a polynomial which  increases in $[0,1]$, satisfies $g(0)=0$ and $g(1)=1$ and satisfies that $g(x) <x$ for each $x \in (0,1)$. Then:
%\begin{enumerate}
%\item[\rm i)]
%%\item\label{lemma:exponential_case}
%If $g^{\prime}(0) <1$ and $g^{\prime}(1)>1$, then
%$$
%\int_0^1 g^{(n)}(x) \ dx < C \, \delta^n\, ,\quad\text{where $C >0$ and $\delta \in (0,1)$.}
%$$
%
%\item[\rm ii)]%\label{lemma:potential_case}
%In general,
%\begin{equation}
%\label{equation:potential_case}
%\int_0^1 g^{(n)}(x) \ dx < C \, \frac{1}{n^{\eta}}\, ,\quad\text{where $C >0$ and $\eta >0$.}
%\end{equation}
%\end{enumerate}
%\end{lemma}

\begin{lemma}
Let $g$ be a polynomial which  increases in $[0,1]$, satisfies $g(0)=0$ and $g(1)=1$, $g(x) <x$ for each $x \in (0,1)$ and $g'(0)=0$ and $g'(1)\ge 1$. Then:
\begin{enumerate}
\item[\rm i)]
%\item\label{lemma:exponential_case}
If $g^{\prime}(1)>1$, then
$$
\int_0^1 g^{(n)}(x) \ dx < C \, \delta^n\, ,\quad\text{where $C >0$ and $\delta \in (0,1)$.}
$$

\item[\rm ii)]%\label{lemma:potential_case}
In general,
\begin{equation}
\label{equation:potential_case}
\int_0^1 g^{(n)}(x) \ dx < C \, \frac{1}{n^{\eta}}\, ,\quad\text{where $C >0$ and $\eta >0$.}
\end{equation}
\end{enumerate}
\end{lemma}
\begin{proof}
i) \textit{The case $g^{\prime}(1)>1$}. Let $\beta$  be a number smaller than $1/2$, but so close to $1/2$ that the segments from $(0,0)$ to $(1/2, \beta)$ and from $(1/2, \beta)$ to $(1,1)$ both lie within the region delimited by the graph of $g$ and the bisectrix of the first quadrant. Let $f$ be the function from $[0,1]$ onto $[0,1]$ whose graph is given by the two segments above.   See Figure \ref{fig:functionf}.
\begin{figure}[h]
\centering\resizebox{6.3cm}{!}{\includegraphics{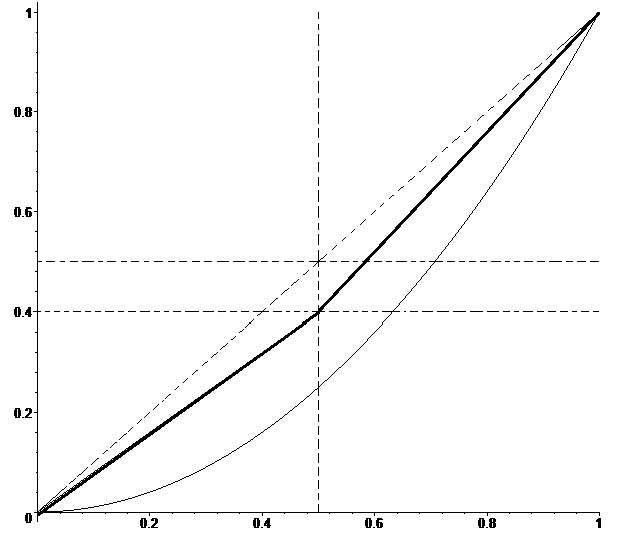}}
\caption{The function $f$.}\label{fig:functionf}
\end{figure}

Observe that $f$ is a bijection from $[0,1]$ onto $[0,1]$ and that for any $0 <x <1$ one has that
$$
g\big(g(x)\big) < g\big(f(x)\big) <f\big(f(x)\big)\, ,
$$
because $g$ is increasing and since for any $x\in (0,1)$, we have $g(x) <f(x)$. In general $g^{(n)}(x) < f^{(n)}(x)$, for any $x \in (0,1)$ and any integer $n \ge 1$.

\smallskip
Let $\alpha_0=1/2$. We define now a sequence indexed by $\mathbb{Z}$ by $\alpha_n=f^{(n)}(\alpha_0)$, for $n \ge 1$, and $\alpha_{-n}=f^{-n}(\alpha_0)$, for $n \ge 1$. It is easy to check that
$$
\alpha_n=\frac{1}{2}\, (2\beta)^n\quad\text{and}\quad 1-\alpha_{-n}=\frac{1}{2}\,\frac{1}{(2-2\beta)^n}.
$$

%Observe, simply by similarity of triangles, that $\alpha_n < C \delta^n$, and $1-\alpha_{-n} < C \delta^n$, for some $C >0$, $\delta \in (0,1)$ and for all $n \ge 1$.
Now, for each $n \ge 1$ we have
$$
f^{(2n)}\big([0,\alpha_{-n}]\big)\subset [0,\alpha_n]\,,
$$
and then that
\begin{align*}
\int_0^1 f^{(2n)}(t)dt &=\int_0^{\alpha_{-n}} f^{(2n)}(t)dt+\int_{\alpha_{-n}}^1 f^{(2n)}(t)dt  \le \alpha_n \alpha_{-n}+(1-\alpha_{-n})
\\
&=
\frac{1}{2}\, (2\beta)^n\, \Big(1-\frac{1}{2}\,\frac{1}{(2-2\beta)^n}\Big) +\frac{1}{2}\,\frac{1}{(2-2\beta)^n}
\le \frac{1}{2}\, (2\beta)^n +\frac{1}{2}\,\frac{1}{(2-2\beta)^n}
<\delta^n
%\le 2C \delta^n\, ,
\end{align*}
taking $\delta=\max(2\beta, 1/(2-2\beta))$.

\smallskip
ii) \textit{The case $g^{\prime}(1)=1$}.  For simplicity, we will change the roles of the points 0 and 1. So assume that $g^{\prime}(0)=1$ and $g'(1)=0$. For some positive integer $k$, and for some $a$ positive and small enough, we have that
$$
g(x) < x(1-ax^k):=f(x), \quad \text{for} \ x \in (0, 1/2]\, .
$$
Following the argument in part i), to obtain \eqref{equation:potential_case}, we just have to analyze the rate of convergence to 0 of the decreasing sequence defined by $\alpha_0=1/2$, and
$$
\alpha_n= \alpha_{n-1} (1-a \alpha_{n-1}^k)\, , \ \text{for} \ n \ge 1\, .
$$
Now, since the sequence $z_n=a \alpha_n^k$ verifies $z_n < z_{n-1} (1-z_{n-1})$, for each $n \ge 1$, we have that
$$
\frac{1}{z_n}-\frac{1}{z_{n-1}} \ge 1
$$
and, consequently, that $z_n <{C}/{n}$ and that $\alpha_n <{C}/{n^{1/k}}$, for some constant $C >0$ and each~$n \ge 1$.
\end{proof}

%\subsection{Strong law}

\section{Comparison with (linear) limit theorems}\label{section:LGN}

For the sake of comparison we now recast  the Weak Law of Large Numbers  and the Central Limit Theorem in the framework of Theorem \ref{main th Sperner}. %{theorem:limittheorem}.
We do not strive for sharp hypothesis. See, for instance, Chapter 9 of \cite{Breiman} and also \cite{anshel}. %\ojo{LA REF \cite{anshel} VA DE OTRA COSA}

\smallskip
Let $\vg=(\gamma_1, \ldots, \gamma_n)$ be a vector in $\mathbb{R}^n$ with (strictly) \textit{positive} coordinates.
Consider the linear function(al) $S_{\vg}: \mathbb{R}^n \mapsto \mathbb{R}$ given by $S_{\vg}(x_1, \ldots, x_n)=\sum_{j=1}^n \gamma_j x_j$; of course, this continuous function $S_{\vg}$ is not as selector.

We are interested in the asymptotic behavior of $\mathbf{S}_{\vg}^{(N)}(X)$ as $N \to \infty$.

\subsection{Weak law}

Here we assume that $ \|\vg\|_1=1$, so that $S_{\vg}$ is an average. Observe that $\|\vg\|_2 <1$.

\smallskip
Assume that $X$ has finite variance  $\sigma^2$ and expectation $\mu$. Observe that $\mathbf{S}_{\vg}(X)$ has  variance $\|\vg\|_2^2\,\sigma^2$ and expectation~$\mu$. In general, $
\mathbf{S}_{\vg}^{(N)}(X)
$ has variance $\|\vg\|_2^{2N} \sigma^2$ and expectation~$\mu$. We conclude, since $\|\vg\|_2 <1$,  that
$$
\mathbf{S}_{\vg}^{(N)}(X)\underset{N \to \infty}{\stackrel{\rm d}{\longrightarrow}}\mu\, .
$$
This is, of course, a rephrasing (of some form) of the  Weak Law of Large Numbers.

\smallskip
Observe that,  since the variance of $\mathbf{S}_{\vg}(X)$ is $\|\vg\|_2^2 \,\sigma^2$ and $\|\vg\|_2^2<1$, the only variables $X$  (with finite variance) such that $\mathbf{S}_{\vg}(X)\stackrel{\rm d}{=}X$ are the constants. Compare with Remark \ref{remark:fixed points of H}.

More generally, let $\vg^{(k)}$, $ k\ge 1$, be a sequence of vectors in $\mathbb{R}^n$ with positive coordinates and such that $\|\vg^{(k)}\|_1=1$ for $k\ge 1$. For such a sequence we have  that, if $\sum_{k=1}^\infty (1-\|\vg^{(k)}\|_2) =+\infty$, then for any random variable $X$ with finite variance,
$$
\big(\mathbf{S}_{\vg^{(N)}}\circ \mathbf{S}_{\vg^{(N-1)}}\circ \cdot \circ\mathbf{S}_{\vg^{(1)}}\big)(X) \underset{N \to \infty}{\stackrel{\rm d}{\longrightarrow}}\E(X)\, .
$$
Observe that if $\sum_{k=1}^\infty (1-\|\vg^{(k)}\|_2)<+\infty$, the limit of $\big(\mathbf{S}_{\vg^{(N)}}\circ \mathbf{S}_{\vg^{(N-1)}}\circ \cdot \circ\mathbf{S}_{\vg^{(1)}}\big)(X)$, if it exists, will not be a constant (unless $X$ itself is a constant).

\subsection{Central limit} Now we assume that $\|\vg\|_2=1$. Observe that $\|\vg\|_3<1$.

Let $X$ be a random variable with $\E(X)=0$, $\E(X^2)=1$ and $\E(|X|^3)=\rho <+\infty$.
The Berry--Esseen inequality gives that
$$
|F_{\mathbf{S}_{\vg}(X)}(x)-\Phi(x)| \le \rho \|\vg\|^3_3\, , \quad \mbox{for any $x \in \mathbb{R}$}\, .
$$
Since $\mathbf{S}_{\vg}^{(N)}$ is also a $\mathbf{S}$ operator but with the vector $\{\big(\gamma_{i_1},\dots, \gamma_{i_N}\big); 1 \le i_j \le n, 1 \le j \le N\}$ instead of the original $\{\gamma_i; 1 \le i \le n\}$, it follows that  for any $N \ge 1$,
$$
|F_{\mathbf{S}_{\vg}^{(N)}(X)}(x)-\Phi(x)| \le \rho \|\vg\|^{3N}_3\, , \quad \mbox{for any $x \in \mathbb{R}$}\, .
$$

Since $\|\vg\|_3 <1$, we conclude that
$$
\mathbf{S}_{\vg}^{(N)}(X) \underset{N \to \infty}{\stackrel{\rm d}{\longrightarrow}} \mbox{standard normal}\, ;
$$
again, a rephrasing of (some form) of the Central Limit Theorem.

Observe that as a consequence of this limit theorem it follows that if $X$ has $\E(X)=0$, $\E(X^2)=1$ (and $\E(|X|^3)<+ \infty$) and if $X$ is a fixed point of $\mathbf{S}_{\vg}$, in the sense that $\mathbf{S}_{\vg}(X)\stackrel{\rm d}{=} X$,  then $X$ is a standard normal variable. Compare with Remark \ref{remark:fixed points of H}.

More generally, let $\vg^{(k)}$, $k\ge 1$, be a sequence of vectors in $\mathbb{R}^n$ with positive coordinates and such that $\|\vg^{(k)}\|_2=1$ for $k\ge 1$. For such a sequence we have  that if $\sum_{k=1}^\infty (1-\|\vg^{(k)}\|_3) =+\infty$ then for any random variable $X$ with $\E(X)=0$, $\E(X^2)=1$ and $\E(|X|^3)=\rho <+\infty$ the following convergence holds
$$
\big(\mathbf{S}_{\vg^{(N)}}\circ \mathbf{S}_{\vg^{(N-1)}}\circ \cdot \circ\mathbf{S}_{\vg^{(1)}}\big)(X) \underset{N \to \infty}{\stackrel{\rm d}{\longrightarrow}}\mbox{standard normal}\, .
$$

\

\noindent\textsc{Francisco Durango:} Departamento de Matem\'{a}ticas, Universidad Aut\'{o}noma de Madrid, 28049-Madrid, Spain.
\texttt{fra.durango@estudiante.uam.es}

\medskip

\noindent\textsc{Jos\'{e} L. Fern\'{a}ndez:} Departamento de Matem\'{a}ticas, Universidad Aut\'{o}noma de Madrid, 28049-Madrid, Spain.
\texttt{joseluis.fernandez@uam.es}

\medskip

\noindent\textsc{Pablo Fern\'{a}ndez:} Departamento de Matem\'{a}ticas, Universidad Aut\'{o}noma de Madrid, 28049-Madrid, Spain.
\texttt{pablo.fernandez@uam.es}

\medskip

\noindent\textsc{Mar\'{\i}a J. Gonz\'{a}lez:} Departamento de Matem\'{a}ticas, Universidad de C\'{a}diz, 11510-Puerto Real, C\'{a}diz, Spain.
\texttt{majose.gonzalez@uca.es}
\end{document}